\documentclass{amsart}

\usepackage{xcolor}
\usepackage{amsthm}
\usepackage{float}
\usepackage{bbm}
\usepackage{mathtools}

\usepackage{array}
\usepackage[caption=false,font=normalsize,labelfont=sf,textfont=sf]{subfig}
\usepackage{textcomp}
\usepackage{stfloats}
\usepackage{url}
\usepackage{verbatim}
\usepackage{graphicx}
\usepackage[colorlinks=true, citecolor=blue, filecolor=black,  urlcolor=blue]{hyperref}
\usepackage{cite}
\usepackage{todonotes}

\newcommand{\rb}[1]{#1} 
\newcommand{\bjz}[1]{#1} 
\newcommand{\R}{\mathbb{R}}

\newtheorem{theorem}{Theorem}[section]
\newtheorem{lemma}[theorem]{Lemma}
\newtheorem{proposition}[theorem]{Proposition}

\theoremstyle{definition}
\newtheorem{definition}[theorem]{Definition}
\newtheorem{example}[theorem]{Example}

\theoremstyle{remark}
\newtheorem{remark}[theorem]{Remark}

\numberwithin{equation}{section}

\DeclareMathOperator*{\argmin}{argmin}

\begin{document}

\title{Proximal Optimal Transport Divergences}

\author[R. Baptista]{Ricardo Baptista}
\address{Computing and Mathematical Sciences, California Institute of Technology, Pasadena CA 91125, United States}
\email{rsb@caltech.edu}
\thanks{}

\author[P. Birmpa]{Panagiota Birmpa}
\address{ Actuarial Mathematics \& Statistics,  School of Mathematical \& Computer Sciences,
Heriot-Watt University and Maxwell Institute, Edinburgh EH14 4AS, United Kingdom}
\email{P.Birmpa@hw.ac.uk}
\thanks{RB is supported by the von K\'{a}rm\'{a}n instructorship at Caltech and a Department of Defense Vannevar Bush Faculty Fellowship (award N00014-22-1-2790) held by Andrew M. Stuart. MK, LR-B and BZ are partially funded by AFOSR grant FA9550-21-1-0354. MK and LR-B are
partially funded by NSF DMS-2307115.}

\author[M. A. Katsoulakis]{Markos A. Katsoulakis}
\address{ Department of Mathematics \& Statistics, University of Massachusetts Amherst, Amherst, MA 01003, United States}
\email{markos@umass.edu}
\thanks{}
\author[L. Rey-Bellet]{Luc Rey-Bellet}
\address{ Department of Mathematics \& Statistics, University of Massachusetts Amherst, Amherst, MA 01003, United States}
\email{luc@math.umass.edu}
\thanks{}

\author[B. J. Zhang]{Benjamin J. Zhang}
\address{Division of Applied Mathematics, Brown University, Providence, RI 02912, United States}
\email{benjamin\_zhang@brown.edu}
\thanks{}





\begin{abstract}
We introduce the proximal optimal transport divergence, a novel discrepancy measure that interpolates between information divergences and optimal transport distances via an infimal convolution formulation. This divergence provides a principled foundation for optimal transport proximals and proximal optimization methods frequently used in generative modeling. We explore its mathematical properties, including smoothness, boundedness, and computational tractability, and establish connections to primal-dual formulations and adversarial learning. The proximal operator associated with the proximal optimal transport divergence can be interpreted as a transport map that pushes a reference distribution 
toward the optimal generative distribution, which approximates the target distribution 
that is only accessible through data samples. Building on the Benamou-Brenier dynamic formulation of classical optimal transport, we also establish a dynamic formulation for proximal OT divergences. The resulting dynamic formulation is a first order mean-field game whose optimality conditions are governed by a pair of nonlinear partial differential equations: a backward Hamilton-Jacobi equation and a forward continuity equation. Our framework generalizes existing approaches while offering new insights and computational tools for generative modeling, distributionally robust optimization, and gradient-based learning in probability spaces. 

\end{abstract}

\maketitle



\section{Introduction}
Divergences and metrics are  tools for measuring the dissimilarity between probability distributions, enabling the comparison of models with data. They are fundamental in data science and machine learning, where tasks such as sampling from unnormalized distributions or generative modeling depend on selecting a suitable discrepancy measure for learning models. The two primary classes of measures for comparing probability distributions are information divergences and optimal  transport distances. Information divergences, such as the Kullback-Leibler (KL) and $f$-divergences, are based on expectations of likelihood-related quantities between distributions. Their ease of computation and optimization through only samples of a target distribution makes them widely used in machine learning. However, these divergences are meaningful only when the distributions share the same support. In contrast, optimal transport distances do not require overlapping supports, making them more flexible in certain applications. However, this advantage comes at the cost of significantly higher computational complexity.

In this paper, we introduce the \emph{proximal optimal transport (OT) divergence} between two probability measures $P$ and $Q$ on some space $X$ and $Y$, denoted by $\mathfrak{D}^c_\varepsilon$   and defined as an infimal convolution
\begin{equation}\label{def: for the intro}
\mathfrak{D}^c_\varepsilon(P\|Q) = \inf_{R}\left\{ T_c(P,R) + \varepsilon \mathfrak{D}(R\|Q) \right\}
\end{equation}
where the infimum is over all probability distributions $R$ on $Y$, $T_c$ is the optimal transport cost corresponding to some cost function $c:X\times Y\to[0,\infty)$ and $\mathfrak{D}$ is an information divergence. The positive parameter $\varepsilon$ controls the trade-off between the information divergence and the optimal transport cost. Under mild assumptions, this is an information divergence. Infimal convolutions have been previously utilized to define novel forms of divergences based on the dual representation of the $1$-Wasserstein distance. For instance, the recently introduced $(f,\Gamma)$-divergence, \cite{dupuis2022formulation,birrell2022f}, combines the dual form of the $1$-Wasserstein distance with $f$-divergences, demonstrating the computational ease and flexibility of $f$-divergences with the advantages of the $1$-Wasserstein distance. This paper extends that concept to encompass general optimal transport distances. Importantly, this work reveals that the proximal OT divergence is the foundational mathematical principle behind optimal transport proximals and proximal optimization methods that are frequently employed in generative modeling. By framing proximal OT divergences in the context of generative modeling, we gain deeper insights that allow for the expansion of definitions, properties, and algorithms, encompassing both primal and dual formulations and extending beyond them.  Despite its fundamental role, the proximal OT divergence has not been explicitly introduced as an independent concept until now.

\subsection{Why proximal optimal transport divergences?}

Proximal OT divergences possess several appealing mathematical properties that make them particularly useful for comparing data to models or for comparing two datasets, even when their supports lie on different manifolds. These properties enable robust and stable learning of target distributions, especially in settings where the data is supported on lower-dimensional spaces. \rb{We outline the key properties and inherited computational benefits of the divergences below.}
\medskip
\begin{itemize}
\item{\bf  Proximal OT divergences are finite:} The divergences are bounded above by the optimal transport distance and the information divergences, i.e., 
\begin{equation}
\label{eq:upper.bound_intro_0}
\mathfrak{D}^c_{\varepsilon}(P\|Q)  \le \min \left\{  
T_{c}(P,Q),  \varepsilon \mathfrak{D}(P\|Q) \right\}.
\end{equation}
This bound ensures finiteness even in the absence of absolute continuity, which enhances the suitability of the divergences for practical applications. For instance, in generative modeling where $P$ represents the generative distribution and $Q$ is the unknown target distribution accessible only through samples, the proximal OT divergence remains finite. In contrast, any information divergence relying on the likelihood ratio $dP/dQ$ may be infinite in such settings as $P$ and $Q$ can be mutually singular.
\medskip
\item{\bf Transport and re-weighting property:} The interpolation via the infimal convolution given in~\eqref{def: for the intro} can be interpreted as a two-step procedure. The first step involves optimally transporting the probability distribution $P$ into an intermediate probability measure $R$ that is absolutely continuous with respect to $Q$ with an associated cost given by $T_c(P,R)$. In the second step, $R$ is re-distributed to $Q$ using the information divergence, incurring a cost $\mathfrak{D}(R\|Q)$. The optimal intermediate measure denoted by $R^*$ minimizes the total cost and determines the optimal proximal OT divergence in~\eqref{def: for the intro}. This property makes them powerful tools for learning manifolds \cite{loaizaganem2024deep_manifold}, as evidenced by related proximal concepts~\cite{birrell2022f, birrell2023adversarially, pidstrigach2022score, zhang2024wasserstein}.
\medskip
\item{\bf Smoothness:} Another key property is that for uniformly continuous cost functions $c$ and $f$-divergences in~\eqref{def: for the intro},  the variational derivative of $\mathfrak{D}^c_{f,\varepsilon}(P\|Q)$ with respect to $P$ remains well-defined without requiring absolutely continuous perturbations relative to  $Q$ (e.g. singular ones), which is primarily treated as the target distribution in this paper. This contrasts with classical KL or  $f$-divergences, which impose such restrictions. This aspect is critical, as variational derivatives play a central role in machine learning tasks, such as the construction of loss functions, generative adversarial networks (GANs), optimal transport methods, and Wasserstein gradient flows, ensuring stability under general conditions. 
\medskip
\item {\bf Transport optimal operator:} Under mild assumptions, we establish uniqueness of the optimal intermediate measure $R^*$ solving \eqref{def: for the intro}. \rb{We formulate the intermediate measure as the output of the following \emph{transport proximal operator} acting on the measure $P$}, i.e.,
\begin{equation}\label{def:Transport proximal operator for intro}
    R^* = \textbf{prox}^c_{\varepsilon \mathfrak{D}(\cdot\|Q)}(P) :=  \argmin_{R\in \mathcal{P}(X)}\left\{T_c(P,R) + \varepsilon \mathfrak{D}(R\|Q)  \right\}.
\end{equation}
This formulation aligns with the broader concept of proximal operators \cite{parikh2014proximal}, reinforcing its theoretical significance and practical utility in generative modeling \cite{zhang2024wasserstein}. The uniqueness of $R^*$ and the transport proximal operator are particularly relevant for generative modeling, where the objective is to approximate an unknown data distribution
$Q$ that is prescribed using only samples. A common approach is to start with a reference distribution $P$ that is easy to sample from, and then to optimally transport it to obtain $R^*$; see second bullet above. \rb{When $\varepsilon$ is large, the divergence will dominate in~\eqref{def:Transport proximal operator for intro} and \(R^*\approx Q\). Thus, the distribution $R^*$ can serve as a generative model to approximate the data distribution. 
This perspective was used in~\cite{onken2021ot,finlay2020train, gu2024combining} to learn a regularized neural ODEs for the generative model by using the proximal operator in~\eqref{def:Transport proximal operator for intro} with the Kullback-Leibler divergence for $\mathfrak{D}$.}
\medskip
\item {\bf Computational aspects:} \rb{The proximal OT divergences have primal and dual formulations that yield computational tools for evaluating the divergence. In particular, the dual formulation offers a variational representation involving test functions.} For example, when the information divergence is the KL-divergence, the dual representation is given by
\begin{equation}\label{eq:dual for intro}
\mathfrak{D}_{\mathrm{KL},\varepsilon}^{c}(P\|Q)= 
\sup_{\substack{(\phi, \psi) \in C_b(X\times X),\\ \phi(x)+\psi(y)\leq c(x,y)}}
\left\{ \mathbb{E}_P[\phi]  -\varepsilon \ln \mathbb{E}_Q[e^{-\psi/\varepsilon}]\right\},
\end{equation}
where $C_b(X\times Y)$ is the set of continuous and bounded functions on $X\times Y$. This representation introduces discriminators, drawing an analogy to adversarial learning methods, \cite{goodfellow2014generative,f-GAN,Wasserstein:GAN}. From a computational perspective, the dual formulation enables efficient implementations using convex neural networks \cite{bengio2005convex}. Moreover, the computation of the variational derivative of proximal OT divergences  involves the discriminator, defined as the optimal $\phi^*$ in \eqref{eq:dual for intro}, which allows us to build gradient flows of the proximal OT divergence, and the associated gradient descent algorithm in probability space \cite{glaser2021kale, Gu_GPA, chen2024learning}.
\medskip
\item {\bf Dynamic formulation of proximal OT divergences:} The primal formulation along with the Benamou-Brenier formulation \cite{Benamou2000}, can be interpreted through the lens of mean-field games (MFG) \cite{lasry2007mean}, in which the MFG framework provides an analogue Benamou-Brenier formula for proximal OT divergences. While deeper connections to Hamilton-Jacobi-Bellman (HJB) theory and MFG-based theoretical analysis could provide further insights, we do not fully explore them here. The MFG formulation enables the use of generative modeling tools \cite{zhang2023mean},  to compute the proximal OT divergence and optimal intermediate measure in the primal formulation \eqref{def: for the intro}, such as generative flows like neural ODEs \cite{NEURALODE} and  OT-flows~\cite{onken2021ot}.
\end{itemize}

\subsection{Related work}
Special cases of the proximal OT divergences have already been extensively analyzed in both theoretical and
computational contexts. In particular, when $T_c$ is the $1$-Wasserstein distance (i.e., $c(x,y) = d(x,y)$) and $\mathfrak{D}$ is a $f$ or R\'{e}nyi divergence, the proximal OT divergence recovers the $(f,\Gamma)$-divergence where $\Gamma$ is the set of all 1-Lipschitz functions or the 1-Wasserstein regularized R\'{e}nyi divergence, that were introduced in \cite{dupuis2022formulation,birrell2022f} and \cite{birrell2023functionspace}, respectively. In general, these divergences interpolate between integral probability metrics (IPMs) over function spaces $\Gamma$, and $f$ and R\'{e}nyi  divergences. An interesting case is the maximum mean discrepancy (MMD) regularized $f$ or  R\'{e}nyi divergence where $\Gamma$ is the unit ball in a reproducing kernel Hilbert space (RKHS), which was studied in \cite{birrell2022f}  and \cite{birrell2023functionspace} and relates to early foundational work on the dual representation of MMDs developed in \cite{Nguyen_IEEE_2010}. Another closely related class of divergences to proximal OT divergences is the Moreau–Yoshida $f$-divergences developed in \cite{terjek2021moreau}. 

All of the aforementioned divergences, both in their general form and specific cases, have been extensively used as objective functionals in generative adversarial networks (GANs) and gradient flows, showing their ability for stable training across a broad range of target distributions \cite{birrell2022structure, Gu_GPA,  birrell2023functionspace, glaser2021kale}, including those with heavy tails \cite{chen2024learning} and those supported on a lower dimensional manifold. In particular, \cite{gu2024combining} applies the combination of 1-Wasserstein and 2-Wasserstein proximals to stably learn distributions supported on manifolds. Optimal transport proximals have also been instrumental in numerical methods for nonlinear partial differential equations and gradient flows, notably in the celebrated Jordan-Kinderlehrer-Otto (JKO) scheme \cite{jordan1998variational}, which employs the Wasserstein proximal operator of the KL divergence. For recent advancements in numerical JKO-based schemes, see \cite{carrillo2022primal}. More recently, optimal transport proximals have attracted significant interest in generative modeling, particularly in the use of 2-Wasserstein regularized KL divergences for training continuous normalizing flows (CNFs) \cite{finlay2020train, onken2021ot}. Approximations of 2-Wasserstein proximals for both linear and nonlinear energy functionals have also been explored in \cite{li2023kernel}, while score-based generative models have been shown to correspond to the 2-Wasserstein proximal of cross-entropy \cite{zhang2024wasserstein, zhang2023mean}. 
As proximal OT divergences can be naturally interpreted as a form of entropic regularization, they share strong connections with Schr\"{o}dinger bridge problems \cite{nutz2024martingale} and various other frameworks based on entropic regularization of optimal transport problems~\cite{chen2016relation, peyre2019computational, chewi2025wasserstein}.

Distributionally robust optimization (DRO) \cite{rahimian2022frameworks} is another related field, which relies on constructing appropriate ambiguity sets to ensure robustness against distributional uncertainty. Common approaches for modeling misspecification include likelihood-based methods, which use information divergences to define the ambiguity set, and distance-based methods, which use Wasserstein distances. Recent advancements have sought to bridge these two perspectives. A novel method in \cite{blanchet2023unifying} extends OT-DRO by incorporating conditional moment constraints, while \cite{birrell2023adversarially} introduces optimal-transport-regularized divergences. The divergences have been employed to enhance the adversarial robustness of deep learning models for applications in privacy and robustness against adversarial attacks. These developments align with special cases of our framework, particularly the finite-dimensional $f$-divergence and KL reformulations. Particular forms of proximal OT divergences have also been applied in Bayes-PAC analysis \cite{guedj2019primer}, resulting in the tractable bounds presented in~\cite{viallard2024tighter_PACBayes}.

\subsection{Organization}
\rb{The remainder of this paper is organized as follows. In Section~\ref{sec:proximalOT} we present the definition and properties of the proximal OT divergence, showing how it interpolates between transport distances and divergences. Section~\ref{sec:duality} uses duality to provide a variational representation of the divergence and Section~\ref{sec:first_variation} leverages the variational representation to compute the first variation of the divergence with respect to perturbations of one input measure. Section~\ref{sec:algorithms} provides the formulations of algorithmic approaches to compute the proximal OT divergence, and Section~\ref{subsect: GF} shows connections to gradient flows. The proofs for all results are relegated to Appendices~\ref{sec:proofs.good.divergence}-\ref{sec:proofs:firstvariation}. Explicit calculations of the proximal OT divergences for univariate and multivariate Gaussian distributions are in Appendix~\ref{sec:Gaussians}.}
\medskip

\paragraph*{\textbf{Notation}} Let $X$ be a Polish space (i.e., a complete, separable metric space), equipped with Borel $\sigma$-algebra $\mathcal{B}_X$ and metric $d: X\times X \to [0,\infty)$. We denote $\mathcal{M}(X)$ to be the set of all finite (signed) Borel measures on $X$ and $\mathcal{P}(X)$ to be the set of all probability measures on $X$. If $X$ and $Y$ are two Polish spaces, and $P\in \mathcal{P}(X)$, $Q \in \mathcal{P}(Y)$, then $\Pi(P,Q)$ is the set of all couplings between $P$ and $Q$, that is, the set of probability measures $\pi \in \mathcal{P}(X\times Y)$ such that $\pi(A\times Y) = P(A)$ and $\pi(X\times B) = Q(B)$ for all Borel sets $A \in \mathcal{B}_X$, $B \in \mathcal{B}_Y$. 
We denote $C(X)$ to be the set of continuous functions and $C_b(X)$ to be the set of bounded continuous functions on $X$, respectively. For functions $\phi\in C(X)$, $\psi \in C(Y)$, denote $\phi\oplus \psi := \phi(x) + \psi(y) \in C(X\times Y)$. 

\section{Proximal optimal transport divergences}
\label{sec:proximalOT}


We now formally introduce the proximal optimal transport (OT) divergences \rb{in Section~\ref{sec:definition}} and establish some key elementary properties that make them useful in computational applications \rb{in Sections~\ref{sec:good.divergence} and~\ref{sec:interpolation}}. Building on the Benamou-Brenier dynamic formulation of optimal transport, we also establish a dynamic formulation in Subsection~\ref{subsec:DF.pOTd} for proximal OT divergences. The resulting dynamic formulation is a first order \emph{mean-field game} \cite{lasry2007mean} whose optimality conditions are governed by a pair of nonlinear partial differential equations, a backward Hamilton-Jacobi (HJ) equation and a forward continuity equation. In this section, we consider general information divergences. In subsequent sections, our results focus specifically on the Kullback-Leibler divergence, though the results extend analogously to any $f$-divergence. 

\subsection{Definition of the proximal OT divergence} \label{sec:definition}
A proximal OT divergence interpolates between an optimal transport distance and an information divergence via an infimal convolution formulation, preserving the advantageous properties of both. An information divergence $\mathfrak{D}(P\|Q)$ provides a notion of ``distance" between probability distributions, satisfying $\mathfrak{D}(P\|Q)\ge 0$ and $\mathfrak{D}(P\|Q)=0$ if and only if $P=Q$.  However, they are generally not considered probability metrics, as they violate symmetry and the triangle inequality. Examples of information divergences include Kullback-Leibler (KL) divergence, and more generally, $f$-divergences.

Given a cost function $c: X \times Y \to \mathbb{R}\cup \{+\infty\}$, an optimal transport cost, (see e.g. \cite{Villani2003topics,Villani2009oldandnew,Santambrogio}), between distributions $P$ and $Q$ is defined as
\begin{equation}\label{eq:transport:cost}
T_{c}(P,Q) =  \inf_{\pi \in \Pi(P,Q)} \mathbb{E}_{\pi}[c],
\end{equation}
where the infimum is taken over all couplings $\pi \in \Pi(P,Q)$, defined as the set of distributions with marginals $P$ and $Q$. Therefore, this formulation seeks the optimal coupling $\pi$ that transports $P$ to $Q$ with minimal average cost with respect to $c$. The definition allows for $X$ and $Y$ to be different spaces, though they are often the same. For simplicity, we will assume that $X=Y$ going forward.  
Let $(X,d)$ be a metric space, then a special class of optimal transport distances is the $p$-Wasserstein distances denoted by $W_p(P,Q)$ for $p\geq 1$, in which $c(x,y) = d(x,y)^p$ (note that when $X=\mathbb{R}^d$, $c(x,y)=|x-y|_p^p$). That is, 
\begin{equation}\label{def:wass.p}
W_p(P, Q) = \inf_{\pi \in \Pi(P, Q)} \left( \int_{\Omega \times \Omega} d(x,y)^p \, d\pi(x, y) \right)^{\frac{1}{p}},
\end{equation}
where $P,Q\in \mathcal{P}_p(X) = \left\{ P \in \mathcal{P} \,:\,  \int d(x,x_0) ^p dP < \infty\right\}$, and $x_0$ is an arbitrary point in $X$. The metric space $\mathcal{P}_p(X)$ endowed with $W_p(P, Q)$ is called the $p$-Wasserstein space. The proximal OT divergence is defined as the \emph{infimal convolution} of the optimal transport distance with the information divergence as given below.

\begin{definition} 
Let $c:X\times Y \to [0,\infty)$ be a cost function, and let $\mathfrak{D}(\cdot\|\cdot)$ be an information divergence. The \emph{proximal optimal transport divergence} of $P\in \mathcal{P}(X)$ with respect to  $Q\in \mathcal{P}(Y)$ is defined as the infimum convolution 
\begin{equation}\label{eq:def:OTdivergence}
\mathfrak{D}^c_\varepsilon(P\|Q) \coloneqq \inf_{R\in \mathcal{P}(Y)}\left\{ T_c(P,R) + \varepsilon \mathfrak{D}(R\|Q) \right\}.
\end{equation}
For $P,Q\in \mathcal{P}_p(X)$ and $T_c(P,Q)=W_p^p(P,Q)$, we use the notation 
\begin{equation}\label{eq:def:WassersteinOTdivergence}
\mathfrak{D}^{p}_{\varepsilon}(P\|Q) \coloneqq \inf_{R\in\mathcal{P}(X)}
\left\{W_p^p(P,R) + \varepsilon \mathfrak{D}(R\|Q) \right\}.
\end{equation}
\end{definition}
\begin{example}{(Examples of proximal OT divergences)}\label{examples of proximal OT divergences}
\begin{enumerate}
\item When $T_c$ is the Wasserstein-$1$ distance, i.e., $c(x,y) = d(x,y)$ and $\mathfrak{D}$ is a $f$-divergence, the formulation recovers the Lipschitz-regularized $f$-divergences. These divergences have been extensively studied for their significance in the design and theoretical analysis of generative models, particularly in the context of stability and robustness in \cite{birrell2022f,dupuis2022formulation,Gu_GPA, chen2024learning}.
\item When $T_c$ is the $p$-Wasserstein distance with $p\geq 2$, and  $\mathfrak{D}$ is the KL-divergence, the formulation leads to the $p$-Wasserstein regularized KL divergence. In Section \ref{subsec:DF.pOTd}, we derive the dynamic formulation of this specific proximal OT divergence through its Benamou-Brenier representation. The dynamic formulation yields optimal trajectories that follow straight-line paths with constant velocity \cite{finlay2020train}. This property aligns with the geodesic structure of the $p$-Wasserstein space and has been effectively utilized in generative modeling, particularly in~\cite{onken2021ot} and is discussed in Section \ref{sec:algorithms.MFG}.
\item Recently, the novel approach in~\cite{gu2024combining} combines 1-Wasserstein and 2-Wasserstein proximals, leading to a proximal OT divergence where $T_c$ corresponds to the 2-Wasserstein distance and $\mathfrak{D}$ is a Lipschitz-regularized $f$-divergence. This formulation ensures the well-posedness of generative flows and enables the robust learning of distributions supported on manifolds.
\end{enumerate}
\end{example}



\subsection{Properties of proximal OT divergences}
\label{sec:good.divergence}

We now derive some fundamental properties of the proximal OT divergence \eqref{eq:def:OTdivergence}, showing its desirable and useful characteristics. To proceed, we impose certain assumptions on the optimal transport cost and the information divergence.
\medskip

\noindent
{\bf Assumption A-T}:  The transport cost $T_c(P,Q)$ is a divergence. Moreover, the map 
\begin{equation}
(P,Q) \mapsto  T_c(P,Q)
\end{equation}
is convex and lower semicontinuous in the weak topology (this holds if $c$ is bounded below and lower semicontinuous \cite{Santambrogio}). 

\medskip

\noindent
{\bf Assumption A-D}:  The map
\begin{equation}
(P,Q) \mapsto  \mathfrak{D}(P\|Q)
\end{equation}
is jointly convex and lower semicontinuous in the weak topology.  
Moreover, we assume that the level sets
\begin{equation}
\{R:\, \mathfrak{D}(R\|Q) \le \alpha\},
\end{equation}
for any $\alpha\geq 0$, are pre-compact in the weak topology and strictly convex. 

\smallskip

These assumptions are satisfied by  the KL divergence.  Moreover, under these assumptions, the proximal OT divergence is indeed a divergence, inheriting the basic convexity and lower semicontinuity from the original information divergence. In particular, compactness and strict convexity guarantee the existence of a unique minimizer $R^*$ in the infimal convolution of the proximal OT divergence. These results are stated in the next two theorems, which are proved in Appendix~\ref{sec:proofs.good.divergence}. 

\begin{theorem}[\bf Divergence properties]\label{thm:divergence}
If assumptions {\bf A-T} and {\bf A-D} are satisfied, then $\mathfrak{D}^c_\varepsilon$ has the following properties. 
\begin{enumerate}
\item{\bf Divergence:} 
$\mathfrak{D}^{c}_{\varepsilon}(P\|Q)$ is a divergence, i.e., 
\begin{equation}\label{eq:divergence}
\mathfrak{D}^c_{\varepsilon}(P\|Q)   \ge 0 \quad \text{and}  \quad \mathfrak{D}^c_{\varepsilon}(P\|Q)=0 \iff P=Q. 
\end{equation}
\item{\bf Convexity and lower semicontinuity:} The map 
\begin{equation}
(P,Q) \mapsto \mathfrak{D}^c_\varepsilon(P\|Q)
\end{equation}
is convex and lower semi-continuous in the weak topology.  
\end{enumerate}
\end{theorem}


\begin{theorem}[\bf Proximal properties]\label{thm:proximal}
If assumptions {\bf A-T} and {\bf A-D} are satisfied, then $\mathfrak{D}^c_\varepsilon$ has the following properties. 
\begin{enumerate}
\item{\bf Upper bound:} The proximal OT divergence is bounded above by the optimal transport distance and the information divergences, i.e., 
\begin{equation}\label{eq:upper.bound}
\mathfrak{D}^c_{\varepsilon}(P\|Q)  \le \min \left\{  
T_{c}(P,Q),  \varepsilon \mathfrak{D}(P\|Q) \right\}. 
\end{equation}
\item{\bf 
Proximal minimizer:} If $\mathfrak{D}^c_\varepsilon(P\|Q) < \infty$, then  there is a unique minimizer $R^* \in \mathcal{P}(Y)$ of \eqref{eq:def:OTdivergence} such that
\begin{equation}\label{eq:minmizer}
\mathfrak{D}^c_\varepsilon(P\|Q)= T_c(P,R^*) + \varepsilon \mathfrak{D}(R^* \|Q) \,.
\end{equation}    
\end{enumerate}
\end{theorem}
 The upper bound \eqref{eq:upper.bound} follows directly by simply setting $R=P$ and $R=Q$ in \eqref{eq:def:OTdivergence}. As an important consequence,
the proximal divergence is finite whenever either the divergence $\mathfrak{D}(P\|Q)$ or the optimal transport is finite. Notably, this holds without requiring absolute continuity, a crucial distinction from the classical KL divergence, where finiteness typically depends on this condition. For example, when considering the $p$-Wasserstein metrics with the KL-divergence, we observe $\mathfrak{D}^p_{\mathrm{KL},\varepsilon}(P\|Q)$ is finite whenever $P$ and $Q$ have $p$-th finite moments without any assumption on absolute continuity of $P$ with respect to $Q$. This property makes proximal OT divergences well-suited for comparing data to models or between datasets, including when data distributions are concentrated on low-dimensional structures, as suggested by the ``manifold hypothesis" in machine learning.

\rb{Furthermore, the proximal OT divergences yields a coarser topology that admits a larger set of target distributions beyond absolute continuity.} 
Both the ball based on the KL divergence $B_{\mathrm{KL}}(Q)=\{\mu:\mathfrak{D}_{\mathrm{KL}}(\mu\|Q)\leq 1 \}$ and the ball based on the optimal transport cost $B_{T_c}(Q)=\{\mu:T_c(\mu\|Q)\leq 1 \}$ are subsets of the proximal OT divergence ball $B_{\mathrm{Prox OT}}(Q)=\{\mu:\mathfrak{D}^c_\varepsilon(\mu\|Q)\leq 1 \}$. While $B_{T_c}(Q)$ includes probability distributions that may not be absolutely continuous with respect to $Q$, $B_{\mathrm{KL}}$ excludes them. Meanwhile, $B_{T_c}(Q)$ may fail to capture heavy-tailed distributions (see Example 1 in \cite{chen2024learning}), while $B_{\mathrm{KL}}$ includes them. $B_{\mathrm{Prox OT}}(Q)$ includes both of these types of distributions. This highlights the advantage of proximal OT divergences in enhancing robustness, both in generative modeling tasks and in the construction of ambiguity sets for DRO.


The existence of a unique proximal minimizer follows naturally from the compactness of the level sets of the map $R \mapsto \mathfrak{D}(R\|Q)$ in the weak topology.  The minimizer $R^*$ is crucial for proving many of the results in this paper, and also plays a significant role in applications and algorithms. For example, in generative modeling applications, where $P$ is a reference probability distribution which is easy to sample from and $Q$ is an unknown data distribution accessible only through samples, the probability distribution $R^*$  often serves as the approximate target distribution that is produced by the generative algorithm~\cite{onken2021ot,finlay2020train,gu2024combining}. This insight motivates the definition of a proximal operator that directly maps the reference measure $P$ to the minimizer $R^*$. This transport proximal operator is a generalization of the proximal operator in $\R^d$ \cite{parikh2014proximal}. When the transport metric is the Wasserstein distance, it is known as the Wasserstein proximal operator and has been studied in \cite{li2023kernel}.

\begin{definition}[\bf Transport proximal operator]\label{def:  Transport proximal operator}
Let $Q$ be a fixed probability distribution and let assumptions {\bf A-T} and {\bf A-D} be satisfied. We define the \emph{transport proximal operator} 
\begin{equation}
\textbf{prox}^c_{\varepsilon \mathfrak{D}(\cdot\|Q)}: \mathcal{P}(X) \to \mathcal{P}(Y)
\end{equation}
as the mapping that assigns each measure $P\in\mathcal{P}(X)$ to the minimizer
\begin{equation}\label{def:Transport proximal operator}
    R^* = \textbf{prox}^c_{\varepsilon \mathfrak{D}(\cdot\|Q)}(P) :=  \argmin_{R\in \mathcal{P}(Y)}\left\{T_c(P,R) + \varepsilon \mathfrak{D}(R\|Q)  \right\}.
\end{equation}
When $T_c$ is the $p$-Wasserstein distance, we denote this operator by $\textbf{prox}^p_{\varepsilon \mathfrak{D}(\cdot\|Q)}$ and refer to it as \emph{Wasserstein proximal operator}. When $Q$ is a fixed probability distribution, we abbreviate the notation for the operator to $\textbf{prox}^c_{\varepsilon \mathfrak{D}}$.
\end{definition}

\begin{remark}[Transport and re-weighting]\label{remark:Transport and re-weighting} The proximal minimizer \rb{elucidates how the divergence compares two distributions}: the measure $P$ is first (optimally) transported to the intermediate measure $R^*$, which is absolutely continuous with respect to $Q$, even if $P$ and $Q$ are not mutually absolutely continuous. Moreover, the support of $R^*$ can differ significantly from $P$. The cost incurred by this movement is the optimal transport cost and essentially depends on the distance over which the mass must be moved. Subsequently, $R^*$ is redistributed to $Q$ using the information divergence. In this stage, the support of $R^*$ remains contained within the support of $Q$; any attempt to expand or shift it beyond the support of $Q$ would result in an infinite information divergence. This structural property implies that proximal OT divergences can effectively capture data supported on manifolds~\cite{gu2024combining}. 
Notably, related proximal divergence frameworks have shown strong manifold learning capabilities, as seen in~\cite{birrell2022f, birrell2023adversarially, pidstrigach2022score, zhang2024wasserstein}.  
\end{remark}
\begin{remark}[Connection to the JKO scheme]\label{rem: Connection to JKO scheme} The Jordan-Kinderlehrer-Otto (JKO) scheme is a time-discretization approach for solving the gradient flow associated with an energy functional $\mathcal{E}(\rho)$ in the 2-Wasserstein space of probability measures $\mathcal{P}_2(X)$. It is closely linked to the transport proximal operator, providing a variational framework that iteratively minimizes a combination of the energy functional and a regularizer given by the 2-Wasserstein distance. The continuous-time gradient flow is defined as the \rb{solution of the PDE} 
    \begin{equation}\label{eq:gradient flow JKO}
    \frac{\partial \rho(x,t)}{\partial t}=-\nabla\cdot \left(\rho\nabla\frac{\delta\mathcal{E}(\rho)}{\delta\rho}\right), \qquad \rho(\cdot,0) = \rho_0(\cdot)
    \end{equation}
    where $\rho(x,t)$ is the probability density evolving over time \rb{starting from an initial probability density $\rho_0$}, and $\frac{\delta\mathcal{E}}{\delta\rho}$ is the variational derivative of the energy functional. The right hand side of \eqref{eq:gradient flow JKO} represents the  Wasserstein gradient flow. 
    Instead of directly solving the continuous-time gradient flow in~\eqref{eq:gradient flow JKO}, the JKO scheme discretizes time into small intervals and finds an approximate solution at each time-step. Specifically, the scheme with a discrete time step of size $\tau$ is formulated as an iterative minimization 
    \begin{equation}\label{eq:iterative minimization}
\rho^*(\cdot,t_{n+1})=\argmin_{\rho}\left\{\mathcal{E}(\rho)+\frac{1}{2\tau}W^2_2(\rho^*(\cdot,t_n),\rho)\right\}.
    \end{equation}
    \rb{The continuous-time limit of the discrete scheme as $\tau\to0$ recovers the solution of the gradient flow PDE~\eqref{eq:gradient flow JKO}; see~\cite{jordan1998variational} for the original work and~\cite{ambrosio2005gradient} for a rigorous overview}. Popular examples are the heat equation and the Fokker-Planck equation, which can be seen as the gradient flow of the entropy functional and the entropy plus a potential energy, respectively. In the latter example, by a simple calculation:  
    \[
    \mathcal{E}(\rho)=\int f(\rho(x)) \rho(x)dx+\int V(x) \rho(x)dx
    =\mathfrak{D}_{\textrm{KL}}(\rho\|q)=\int\rho(x)\log\frac{\rho(x)}{q(x)}dx,\]
    with $f(t)=t\log t$ and $q(x)\propto e^{-V(x)}$. 
    
    The connection between the JKO scheme and the proximal OT divergences follows by interpreting~\eqref{eq:iterative minimization} using operator~\eqref{def:Transport proximal operator}, i.e.,
    \[
    \rho^*(\cdot,t_{n+1})=\frac{1}{2\tau}\textbf{prox}^2_{2\tau \mathfrak{D_{\mathrm{KL}}}(\cdot\|q)}(\rho^*(\cdot,t_n))
    \] 
    starting from some initial probability density $\rho(\cdot,0) = \rho_0$.
\end{remark}

\subsection{Proximal OT divergences interpolate between optimal transport distances and information divergences}
\label{sec:interpolation}

The parameter $\varepsilon$ plays a crucial role in proximal OT divergences. It controls the trade-off between the information divergence term and the optimal transport cost, thereby determining the relative importance of reweighting versus transport. More precisely, $\varepsilon$ scales the information divergence term, affecting how much it contributes to the total cost. For example, in the divergence-dominant regime, i.e., when $\varepsilon$ is large, the information divergence term becomes more significant and even a small value of divergence $\mathfrak{D}(R\|Q)$ contributes significantly to the overall objective as it is magnified by $\varepsilon$. 
This effect is formalized in the following theorem, demonstrating that the proximal OT divergence is not only a lower bound for both the divergence and the optimal transport (see \eqref{eq:upper.bound}), but also serves as a continuous regularization of both terms.

\begin{theorem}[\bf Interpolation properties]\label{thm:interpolation}
Under assumptions {\bf A-T} and {\bf A-D}, the following holds:
\begin{enumerate}
 \item{\bf Divergence regularization by OT}.  As $\varepsilon \searrow 0$, $\frac{1}{\varepsilon} \mathfrak{D}^{c}_{\varepsilon}$ is increasing and 
\begin{eqnarray}
    \lim_{\varepsilon \to 0} \frac{1}{\varepsilon} \mathfrak{D}^{c}_{\varepsilon}(P\|Q)  &=& \mathfrak{D}(P\|Q)
\end{eqnarray}
 
\item{{\bf OT regularization by divergence}}. As $\varepsilon \nearrow \infty$, $\mathfrak{D}^{c}_{\varepsilon}$ is increasing and 
\begin{eqnarray}
    \lim_{\varepsilon \to \infty} \mathfrak{D}^{c}_{\varepsilon}(P\|Q) &=&  T_c(P,Q)
    \end{eqnarray}
\end{enumerate}
\end{theorem}

The proof of this theorem is given in Appendix~\ref{sec:proofs.interpolation} and is essentially the same as the one given in \cite{dupuis2022formulation,birrell2022f} for the family of $(f,\Gamma)$-divergences introduced therein.


\subsection{Dynamic formulation of proximal OT divergences and mean-field games}\label{subsec:DF.pOTd}

Here, we demonstrate that the proximal OT divergences have a continuous-time, flow-based formulation that augments the classical Benamou–Brenier formulation of optimal transport \cite{Benamou2000} with a mean field game (MFG) structure. Specifically, let $X=Y=\mathbb{R}^d$ and consider the transport cost  $c(x,y)=|x-y|^p$ for $p>1$. 
According to the Benamou-Brenier formula, see e.g.\thinspace\cite{Santambrogio}, optimal transport can be interpreted as the problem of finding an optimal vector field $v_t : \mathbb{R}^d \rightarrow \mathbb{R}^d$ \rb{for all $t \in [0,1]$} 
that continuously transports the probability measure $P$ into the measure $R$, i.e., 
\begin{align}
    W_p^p(P,R) = \min_{\rho, v} \left\{ \int_0^1 \mathbb{E}_{\rho_t}\left[|v_t|^p\right] d t : \begin{array}{c}\partial_t\rho_t + \nabla \cdot(\rho_t v_t) = 0 \\ \rho_0 = P,\, \rho_1 = R  \end{array}
    \right\}.
\end{align}
This leads immediately to the following theorem that provides an equivalent formulation of the divergence based on dynamic optimal transport. 
\begin{theorem}[\bf Dynamic formulation of proximal OT divergences]\label{thm:dynamic:pot}
    For $p>1$, the proximal OT divergence  $\mathfrak{D}^p_\varepsilon(P\|Q)$
    has the following representation:
\begin{align}
    \mathfrak{D}^p_\epsilon(P\|Q) &=  \inf_{R\in \mathcal{P}(\R^d)}\left\{ \varepsilon \mathfrak{D}(R\|Q) + W_p^p(P,R) \right\} \nonumber \\
    & = \min_{v,\rho} \left\{ \varepsilon \mathfrak{D} (\rho_1\|Q) + \int_0^1 \mathbb{E}_{\rho_t}\left[|v_t|^p\right] dt : 
    \begin{array}{c}
    \partial_t \rho_t + \nabla \cdot(\rho_t v_t) = 0 \\  
    \rho_0 = P
     \end{array}\right\} \label{eq:dynamic:pot}
\end{align}
\end{theorem}

The choice of the terminal time $T$, chosen to be $T = 1$ in Theorem \ref{thm:dynamic:pot}, is entirely arbitrary. A more natural choice is $T=\varepsilon$, and by rescaling time for the flow accordingly, we obtain the divergence
\begin{eqnarray}\label{dynamic.formulation}
\mathfrak{D}^p_T(P\|Q) &=&  T \inf_{R\in \mathcal{P}(\R^d)}\left\{  \mathfrak{D}(R\|Q) + \frac{1}{T} W_p^p(P,R) \right\} \nonumber \\
    & =& T \min_{v,\rho} \left\{  \mathfrak{D} (\rho_T\|Q) + \rb{\frac{1}{T^2} }\int_0^T \mathbb{E}_{\rho_t}\left[|v_t|^p\right] dt : 
    \begin{array}{c}
    \partial_t \rho_t + \nabla \cdot(\rho_t v_t) = 0 \\  
    \rho_0 = P
    \end{array}
    \right\},
    \label{eq:dynamicpOTrescaled}
\end{eqnarray}
\rb{where the extra factor $1/T$ arises from the time-substitution $t \rightarrow Tt$.}
\eqref{eq:dynamicpOTrescaled} results in a more convenient form for time-discretization and learning algorithms. We can interpret the parameter $T$ as the time-horizon  allotted for the Wasserstein component of the proximal OT divergence to carry out the transport.

\begin{remark}[Proximal in generative flows]
 The dynamic formulation of proximal OT divergence is particularly useful in generative modeling algorithms which learn the optimal vector field $v_t^*$ in~\eqref{eq:dynamicpOTrescaled} and thus the flow driving the source $P$ toward the proximal \rb{distribution} $\rho_T^*=R^*=\textbf{prox}^c_{\varepsilon \mathfrak{D}}(P)$. The interpolation properties established in Theorem \ref{thm:interpolation} indicate that $\rho_T^*$ is a good approximation of $Q$ 
when $T>1$ in \eqref{eq:dynamicpOTrescaled}. Specifically, Theorem \ref{thm:interpolation} states that $\lim_{T \to \infty}\mathfrak{D}^{c}_{T}(P\|Q) = W_p^p(P,Q)$. Thus, by the approximation  
\begin{equation}
\mathfrak{D}^{c}_{T}(P\|Q)= T\mathfrak{D}(R^*\|Q) +  W_p^p(P,R^*)\approx W_p^p(P,Q),\qquad T>1.
\end{equation}
Hence, $R^*$ is a good approximation of $Q$, and $\rho^*_T$ serves as a good surrogate of $Q$. 
\end{remark}

\begin{remark}[The mean field-game for proximal OT divergences]
 The dynamical formulation of the proximal OT divergence closely resembles the Benamou-Brenier formulation of the Wasserstein transport cost \cite{Benamou2000}. However, the divergence term and the optimization over the intermediate measure $\rho_T=R$ in \eqref{eq:dynamic:pot} serves as a relaxation of the terminal condition in the Benamou-Brenier formulation. Mathematically, this turns the optimization problem~\eqref{eq:dynamic:pot} into a mean-field game (MFG)~\cite{lasry2007mean}, with $\mathfrak{D} (\rho_T\|Q)$ and $|v|^p$ playing the role of the terminal cost and action cost in the MFG formulation, respectively. The optimality conditions are a coupled system of non-linear PDEs: a backward Hamilton-Jacobi-Bellman (HJB) equation that governs how an individual agent optimally controls its movement while considering the evolving population distribution $\rho(\cdot,t)$, with a forward continuity equation that describes the dynamics of the density of the agents under a particular velocity field~\cite{zhang2023mean}. For example, if the Wasserstein-$2$ distance is chosen to be the optimal transport cost \bjz{and $\varepsilon = 2/T$}, the mean-field game yields the following system of coupled PDEs
\begin{align}
    \begin{dcases}
        - \frac{\partial U}{\partial t} + \frac{1}{2}|\nabla U|^2 = 0 \\
        \frac{\partial \rho}{\partial t} - \nabla \cdot(\rho\nabla U) = 0 \\
        U(x,T) = 1 + \log \frac{\rho(x,T)}{Q(x)}, \, \rho(x,0) = P(x).
    \end{dcases}
\end{align}
Moreover, the optimal velocity is a time-dependent gradient field $v^*(x,t) = -\nabla U(x,t)$. 
\end{remark}

\section{Variational representations for proximal OT divergences}
\label{sec:duality}
In this section, we show that the dual formulation of the proximal OT divergences provides a variational representation of the proximal OT divergences based on a suitable classes of test functions. Variational representations of divergences, integral probability metrics, and optimal transport costs play an important theoretical and practical role. For example, these representations are an essential ingredient in generative adversarial networks
(GANs) which are constructed by solving a minmax game that minimize over parameters of the generator and maximize over the function space in the variational representation of the divergence to identify the discriminator 
\cite{goodfellow2014generative}. 
Here in particular, the dual formulation  allows us to compute proximal OT divergences using convex neural networks as discriminators; see Section \ref{sec:algorithms.CNN} for a complete discussion. \rb{In Section~\ref{sec:dual_representations} we review variational representations of information divergences and optimal transport distances. We then derive variational representations for the proximal OT divergence in Section~\ref{sec:dualityproxOT}, which we use to derive a data processing inequality in Section~\ref{sec:DPI}.}


\subsection{Duality provides alternative representations of optimal transport distances and information divergences} \label{sec:dual_representations}

In this section, we choose the KL-divergence as an information divergence to present our results. One could also consider more general $f$-divergences.  We begin by revisiting some well-known facts about optimal transport and the KL divergence.

\textbf{Duality for optimal transport.}
We summarize some well-known results in optimal transport; see \cite{Villani2003topics, Villani2009oldandnew, Santambrogio} for proofs. Given a cost function $c: X \times Y \to \mathbb{R}\cup \{+\infty\}$, we consider the function space
\begin{equation}
\Phi_c = \left\{ (\phi, \psi) \in C_b(X\times Y), \phi \oplus \psi \le c \right\}
\end{equation}
(recall $\phi \oplus \psi$ denotes the function $\phi(x)+\psi(y)$). A fundamental result in optimal transport is the following duality result: suppose $c$ is bounded below and lower semicontinuous, then we have 
\begin{equation}\label{eq:transport:duality}
\begin{aligned}
T_{c}(P,Q) & =\inf_{\pi \in \Pi(P,Q)} \mathbb{E}_{\pi}[c] = \sup_{ (\phi, \psi) \in \Phi_c} \left\{  \mathbb{E}_P[\phi]+\mathbb{E}_{Q}[\psi]\right\}.
\end{aligned}
\end{equation}
Regarding the existence of optimizers,  an optimal coupling $\pi^*$ always exists given that $X$ and $Y$ are Polish.  If $c$ is uniformly continuous and bounded, 
then there exist optimizers $(\phi^*,\psi^*) \in \Phi_c$.  For general cost functions $c$, we cannot expect the optimizers to be bounded but there exists optimizers $\phi \in L^1(P)$, $\psi \in L^1(Q)$
satisfying $\phi\oplus \psi \le c$ provided we assume that 
\begin{equation}
c(x,y) \le a(x) + b(y) \quad \text{ with } a \in L^1(P), b \in L^1(Q). 
\end{equation}
For example if $c(x,y)=d(x,y)^p$,
this condition holds for all probability measures in the $p$-Wasserstein space defined in~\eqref{def:wass.p}.

\textbf{Duality for the KL-divergence.}
The KL-divergence between two probability measures $R\in \mathcal{P}(Y)$ and $Q\in \mathcal{P}(Y)$ is given by 
\begin{equation}\label{eq:KLdivergence}
\mathfrak{D}_{\textrm{KL}}(R\|Q) = 
\left\{
\begin{array}{cl}
\mathbb{E}_R\left[ \log \frac{dR}{dQ}\right] & \textrm{ if } R \ll Q \\
+\infty \textrm{ otherwise}
\end{array}
\right. \,.
\end{equation}
The Donsker-Varadhan formula provides a variational representation of the KL divergence (see \cite{dupell4} for a proof), which is based on the convex duality between the divergence and the cumulant generating function $\Lambda(\psi)=\log \mathbb{E}_Q[e^{\psi}]$. That is, 
\begin{equation}\label{eq:donsker.varadhan}
\sup_{\psi \in C_b(Y)} 
\left\{ 
\int_Y \psi \,dR- \log \mathbb{E}_Q[e^{\psi}]
\right\} =  
\left\{ 
\begin{array}{cl} 
\mathfrak{D}_{\textrm{KL}}(R\|Q),  & R \in \mathcal{P}(Y) \\
+\infty, &  R \in \mathcal{M}(Y) \setminus \mathcal{P}(Y).
\end{array}
\right. 
\end{equation}
The reader should note that this variational representation implies that we define $\mathfrak{D}_{\textrm{KL}}(R\|Q)$ as $+\infty$ if $R$ is not a probability measure. 
While one could extend  $\mathfrak{D}_{\textrm{KL}}(R\|Q)$ 
to $\mathcal{M}(Y)$ using the definition in the first line \eqref{eq:KLdivergence} to obtain a finite value,  this leads to another convex duality representation (see e.g. \cite{birrell2022f}). However, the representation in \eqref{eq:donsker.varadhan} is the one that will be essential for our purposes. 

We will also use the (convex) dual formula to \eqref{eq:donsker.varadhan} referred to as the Gibbs variational principle:  if $\psi \in C_b(Y)$
then 
\begin{equation}\label{eq:gibbs}
\log \mathbb{E}_Q[e^{\psi}] = \sup_{R \in \mathcal{P}(Y)} 
\left\{ \mathbb{E}_R[\psi] - \mathfrak{D}_{\textrm{KL}}(R\|Q)\right\} ,
\end{equation}
and the supremum is attained when $R$ is such that 
\begin{equation}
\frac{dR}{dQ} = \frac{e^{\psi}}{\mathbb{E}_Q[e^{\psi}]}\,.
\end{equation}

\subsection{Duality for proximal OT divergences via variational representations} \label{sec:dualityproxOT}
In this section, we establish a dual formulation for the proximal OT divergence when KL divergence is combined with an optimal transport cost. This is achieved by leveraging the variational representations given in \eqref{eq:gibbs} and \eqref{eq:transport:duality} respectively. In particular, the infimum convolution formula yields the following duality result. 


\begin{theorem}{\bf(Duality for proximal  OT-divergence)} \label{thm:duality.Polish} 
Suppose $X$ and $Y$ are Polish spaces and that $c$ is lower-semicontinuous and bounded below.  Given  $P\in \mathcal{P}(X)$ and $Q\in \mathcal{P}(Y)$ suppose that $c(x,y)\le a(x)+b(y)$ with $\mathbb{E}_P[a] < \infty$ and $\mathbb{E}_Q[b] < \infty$. Then we have the following variational representations
\begin{equation}\label{eq:OTdivergence:dual}
\begin{aligned}
\mathfrak{D}_{\mathrm{KL},\varepsilon}^{c}(P\|Q)
&= \inf_{R \in \mathcal{P}(Y)} T_c(P,R)+ \varepsilon \mathfrak{D}_{\mathrm{KL}}(R\|Q)
\\
&= \sup_{(\phi, \psi) \in \Phi_c} \left\{ \mathbb{E}_P[\phi] - \varepsilon \log \mathbb{E}_Q[ e^{-\psi/\varepsilon}] \right\}{}
\end{aligned}
\end{equation}
\end{theorem}

The proof of Theorem \ref{thm:duality.Polish}, at this level of generality, is rather long and is given in Appendix~\ref{sec:proofs.duality}.  The general idea is to use a convex analysis argument in the case where $X$ and $Y$ are compact and $c$ is continuous and then prove the general case using several approximation arguments and lower semicontinuity. In this respect, the strategy of the proof follows quite closely the proof of duality for optimal transport given in \cite{Villani2003topics} and use many of the technical estimates and ideas therein. 

For now, we are content with a formal proof based on the dual formulation of the transport cost $T_c(P,R)$ and the  Gibbs variational principle for the KL-divergence. Without rigorous justification, observe that, by exchanging infimum and supremum--- which could likely be justified for compact spaces using an appropriate minimax theorem---
we obtain
\begin{equation}\label{eq:formal proof}
\begin{aligned}
\mathfrak{D}_{\mathrm{KL},\varepsilon}^{c}(P\|Q)&=
\inf_{R\in\mathcal{P}(Y)} \left\{ T_c(P,R) + \varepsilon \mathfrak{D}_{\textrm{KL}}(R\|Q) \right\} \\
&=\inf_{R\in\mathcal{P}(Y)}\sup_{(\phi,\psi)\in \Phi_c}\left\{\mathbb{E}_P[\phi]+\mathbb{E}_{R}[\psi] + \varepsilon \mathfrak{D}_{\textrm{KL}}(R\|Q)\right\} \quad (\textrm{ by } (\ref{eq:transport:duality})) \nonumber \\
&= \sup_{(\phi,\psi)\in \Phi_c} 
\inf_{R\in\mathcal{P}(Y)}
\left\{ 
\mathbb{E}_P[\phi] + \mathbb{E}_{R}[\psi] + \varepsilon \mathfrak{D}_{\textrm{KL}}(R\|Q) \right\} 
\quad 
(\inf \leftrightarrow \sup) 
\nonumber  \\
&= \sup_{(\phi,\psi)\in \Phi_c}
\left\{  
\mathbb{E}_P[\phi] - \varepsilon \sup_{R\in\mathcal{P}(Y)} 
\left\{ 
\mathbb{E}_{R}[-\psi/\varepsilon] -  \mathfrak{D}_{\textrm{KL}}(R\|Q)
\right\} \right\}
\nonumber  \\
&= 
\sup_{(\phi,\psi)\in \Phi_c}
\left\{ \mathbb{E}_P[\phi]  -\varepsilon \ln \mathbb{E}_Q[e^{-\psi/\varepsilon}]\right\}\quad 
(\textrm{by } (\ref{eq:gibbs})). 
\end{aligned}
\end{equation}
\begin{example}{(Lipschitz-regularized $f$-divergences)}
In Example \ref{examples of proximal OT divergences}, the Lipschitz-regularized $f$-divergences \cite{Gu_GPA} arise as a special case of proximal OT divergences, where the information divergence is given by an $f$-divergence and the 1-Wasserstein is the optimal transport cost. It can be readily shown that a  $c$-concave function, defined in Appendix~\ref{sec:proofs.duality}, corresponds exactly to a 1-Lipschitz function, and vice versa.  Moreover, in this case, we have $\psi=\phi^c=-\phi$ (see Lemma \ref{lem:maxDP}) and thus
\begin{equation*}
\mathfrak{D}_{\mathrm{KL}}^{{\mathrm{Lip}}_{\frac{1}{\varepsilon}}(X)}(P\|Q)\coloneqq\sup_{\phi\in {\mathrm{Lip}}_{\frac{1}{\varepsilon}}(X)}
\left\{ \mathbb{E}_P[\phi]  - \ln \mathbb{E}_Q[e^{\phi}]\right\}=\frac{1}{\varepsilon}\mathfrak{D}_{\mathrm{KL},\varepsilon}^{1}(P\|Q),
\end{equation*}
where ${\mathrm{Lip}}_{\frac{1}{\varepsilon}}(X)$ is the set of Lipschitz functions on $X$ with Lipschtz constant $1/\varepsilon$. This can be extended to the dual form of the Lipschitz-regularized $f$-divergences as  $
\mathfrak{D}_{f}^{{\mathrm{Lip}}_{\frac{1}{\varepsilon}}(X)}(P\|Q)=\frac{1}{\varepsilon}\mathfrak{D}_{f,\varepsilon}^{1}(P\|Q)$.
\end{example}
Next, we derive the existence and characterization of the solution $R^*$ to the infimal convolution problem \eqref{eq:def:OTdivergence} as well as the optimizers $(\phi^*,\psi^*)$ in the dual representation \eqref{eq:OTdivergence:dual}. Furthermore, the three optimizers $R^*,\phi^*,\psi^*$ are related in such a way that they reinforce the transport and re-weighting interpretation of proximal OT divergences as discussed in Remark \ref{remark:Transport and re-weighting} and stated in Theorem \ref{thm:duality.optimizers} (d) and further explained thereafter. 

\begin{theorem}{\bf(Optimizers for the proximal OT divergence)} \label{thm:duality.optimizers} 
Assume that $X$ and $Y$ are Polish spaces and that $c$ is lower-semicontinuous and bounded below.  
Let $P\in \mathcal{P}(X)$ and $Q\in \mathcal{P}(Y)$ and assume that $c(x,y)\le a(x)+b(y)$ with $\mathbb{E}_P[a] < \infty$ and $\mathbb{E}_Q[b] < \infty$.

\begin{itemize}
\item[$(a)$] There exists a unique minimizer $R^*$ in  \eqref{eq:def:OTdivergence}.
\item[$(b)$] If $c$ is bounded and uniformly continuous then there exists a maximizer pair $(\phi^*, \psi^*)\in\Phi_c$ given in \eqref{eq:OTdivergence:dual}.  
Up to a constant shift $(\phi^*, \psi^*) \to (\phi^*-a, \psi^*+a)$, 
$\phi^*$ is unique $P$ almost surely and $\psi^*$ is unique $Q$ almost surely. 
\item[$(c)$] In general, the supremum is attained for a pair $(\phi^*, \psi^*) \in L^1(P)\times L^1(Q)$ with the same uniqueness property as for $c$ continuous.  
\item[$(d)$]  The optimizer $R^*$, $\phi^*, \psi^*$ are 
 related in the following manner:
\begin{enumerate}
\item The pair $(\phi^*,\psi^*)$ is the optimizer in the dual representation of the $T_c(P,R^*)$, i.e. 
\begin{equation}\label{eq:optimalR*1}
T_c(P,R^*)= \mathbb{E}_P[\phi^*] + \mathbb{E}_{R^*}[\psi^*]  \,.
\end{equation}
\item The Radon-Nikodym  derivative of $R^*$ with respect to $Q$ is given by 
\begin{equation}\label{eq:optimalR*2}
\frac{dR^*}{dQ} =  \frac{e^{- \psi^*/\varepsilon}}{\mathbb{E}_Q[e^{-\psi^*/\varepsilon}]} 
\end{equation}
\end{enumerate}
\end{itemize}
\end{theorem}
Note that all optimizers $R^*$, $\phi^*, \psi^*$ depend on $\varepsilon$; however, for notational simplicity, we omit this dependence throughout. The proof of this theorem is given in Appendix~\ref{sec:proofs.duality}.  The technical part consists in proving the existence of optimizers.  Here, we focus on explaining the relationship between the optimizers. Assuming we have established both the duality and the existence of maximizers  $(\phi^*,\psi^*)$  \eqref{eq:OTdivergence:dual}, then we have 
\begin{equation}
 \mathbb{E}_{\pi^*}[c] + \varepsilon \mathfrak{D}(R^*\|Q) = \mathbb{E}_P[\phi^*] -\varepsilon \ln \mathbb{E}_Q[e^{-\psi^*/\varepsilon}]
\end{equation}
where $\pi^*$ is the optimal transport plan between $P$ and $R^*$.  Equivalently 
\begin{equation}\label{eq:proofduality}
\underbrace{ \mathbb{E}_{\pi^*}[c - \phi^* - \psi^*]
 }_{\ge 0}+ \varepsilon 
 \underbrace{\left\{ \mathfrak{D}(R^*\|Q) -  \mathbb{E}_{R}[-\psi^*/\varepsilon] + \log \mathbb{E}_Q[e^{-\psi^*/\varepsilon}] \right\}
 }_{\ge 0} = 0
\end{equation}
and, since both terms must vanish, we get exactly the characterizations
of the optimizers in \eqref{eq:optimalR*1} and \eqref{eq:optimalR*2} in Theorem \ref{thm:duality.optimizers}.  

Part $(d)$ in Theorem \ref{thm:duality.optimizers} can be interpreted as a transport and re-weighting process (see also Remark \ref{remark:Transport and re-weighting}) where  \eqref{eq:optimalR*1} is the transport of $P$ to $R^*$ and the emerging Gibbs distribution in \eqref{eq:optimalR*2} describes the re-weighting of $R^*$ to $Q$ with the weighting determined by $\psi^*/\varepsilon$. In generative modeling, this process enables the generation of new data by transporting and re-weighting samples. 

\rb{Lastly, we show that the optimizer $\phi^*$ given in Part $(d)$ of Theorem \ref{thm:duality.optimizers} corresponds to an optimal transport map from $P$ to the intermediate measure $R^*$. To do so,}
we recall the Monge formulation of the optimal transport problem 
between measures $P,R^*\in \mathcal{P}(\mathbb{R}^d)$ 
that seeks a measurable map $T:\mathbb{R}^d\to \mathbb{R}^d$ pushing $P$ onto $R^*$ while minimizing the optimal transport cost: \begin{equation}\label{MongeProblem}
T_{c}(P,R^*) =  \inf_{T} \int_{\mathbb{R}^d}c(x,T(x))dP(x),
\end{equation}
For the special case $c(x,y)=\frac{1}{2}|x-y|^2$, under suitable conditions for the existence of a map (see Theorem 4.2 in \cite{santambrogio2017euclidean}), the optimal transport map is given by  $T^*(x)=x-\nabla\phi^*(x)$ where $\phi^*$ is given in Part $(d)$ of Theorem \ref{thm:duality.optimizers}. 
\medskip
\begin{remark} (Alternative expression for the solution of JKO scheme at each step) The iterative minimization in the JKO scheme can be expressed using the transport proximal operator: 
 \[
    \rho^*(\cdot,t_{n+1})=\frac{1}{2\tau}\textbf{prox}^2_{2\tau \mathfrak{D_{\mathrm{KL}}}(\cdot\|q)}(\rho^*(\cdot,t_n)),
    \] 
see Remark \ref{rem: Connection to JKO scheme}. From Theorem \ref{thm:duality.optimizers}(d)2), the corresponding proximal minimizer satisfies
\begin{equation}
\frac{\rho^*(\cdot,t_{n})}{q} 
=\frac{e^{- \psi_{n}^*/2\tau}}{2\tau\mathbb{E}_Q[e^{- \psi_{n}^*/2\tau}]}\textrm{\;\;which implies \;} \rho^*(\cdot,t_{n})\propto e^{- \psi_{n}^*/2\tau-V}\,.
\end{equation}
This expression aligns with Proposition 8.7 in \cite{Santambrogio}.
\end{remark}

\subsection{Applications of duality}
\label{sec:DPI}
In this section, we derive further properties of the proximal OT divergence through its variational formulation. All proofs can be found in Appendix~\ref{sec:proofs:DPI}. We first show that the proximal $p$-Wasserstein divergence is invariant under metric isomorphism. 
\begin{theorem}{\bf(Invariance under isomorphisms)} \label{thm:isomorphism.invariance} Suppose $T$ is a metric isomorphism of $X$ (i.e. $T:X \to X$ is invertible and preserves distance, i.e., $d(Tx,Ty)=d(x,y)$ for all $x,y \in X$). Then 
\begin{equation}\label{eq:DPIisomorphism}
 \mathfrak{D}^{p}_\varepsilon( T_\# P\| T_\# Q) = \mathfrak{D}^{p}_\varepsilon( P\| Q).
\end{equation}
\end{theorem}

Next, we establish a version of the data processing inequality,\cite{van2014renyi}. Suppose that $K(x,dx')$ is a kernel from $X$ to $X'$ such that for each $x$, $K(x,dx') \in \mathcal{P}(X')$ and $K(x,A)$ is measurable for every $x$ and every Borel set $A\subset X'$. 
The kernel $K$ induces a map that takes functions on $X'$ to functions of $X$ by $K\phi(x)=\int_{X'} K(x, dx') \phi(x')$. Moreover, $K$ maps probability measures on $X$ to probability measures on $X'$ by 
$K_\# P(A) = \int_X K(x,A) dP(x)$. 

\begin{theorem}{\bf (Data processing inequality)}\label{thm:dpi}
Suppose $X=Y$ and $K$ is a kernel from $X$ to $X'$ which maps $C_b(X')$ into $C_b(X)$. Then we have 
\begin{equation}\label{eq:DPIgeneral}
 \mathfrak{D}^{c}_\varepsilon( K_\# P\| K_\# Q) \le  \mathfrak{D}^{K c}_\varepsilon( P\| Q)
\end{equation}
where $Kc$ is the cost function on $C$ defined by 
\begin{equation}\label{eq:Kc}
Kc(x,y) = \int_X \int_Y c(x,y) K(x,dx) K(y,dy).
\end{equation}
\end{theorem}

Finally, we prove that the proximal OT divergence with $p$-Wasserstein $\mathfrak{D}_{p}^\varepsilon$ is additive  for product measures.
\begin{theorem}{\bf (Additivity for product measures)}
\label{thm:additivity} Assume $X=Y=\mathbb{R}^n$ and $c(x-y)=\|x-y\|^p$.  If  $P= P_1 \times P_2$ and $Q=Q_1\times Q_2$  are product measures with $P_1, Q_1 \in \mathcal{P}(\mathbb{R}^{n_1})$ and $P_2, Q_2 \in \mathcal{P}(\mathbb{R}^{n_2})$ then the proximal OT is additive
  \begin{equation}\label{eq:additivity}
  \mathfrak{D}_{\mathrm{KL},\varepsilon}^{p}( P_1 \times P_2  \| Q_1\times Q_2)=   \mathfrak{D}_{\mathrm{KL},\varepsilon}^{p}( P_1  \| Q_1 ) + \mathfrak{D}_{\mathrm{KL},\varepsilon}^{p}( P_2  \| Q_2 )
\end{equation}
\end{theorem}  

The proof relies on both the primal and dual representations of the divergence. By restricting the infimum to product intermediate measures, we obtain one side of the inequality in \eqref{eq:additivity}, while restricting the supremum to functions of the form $\phi(x)=\phi_1(x_1) + \phi_2(x_2)$ (and similarly for $\psi$) yields the reverse inequality. The full details are provided in Appendix~\ref{sec:proofs:DPI}.

\section{The first variation of the proximal OT divergence} \label{sec:first_variation}
 In this section we show that the variational derivative of the proximal OT divergence around a fixed  measure $Q$ exists under very broad conditions and is practically computable. The existence and the evaluation of  variational derivatives of proximal OT divergences
involves the discriminator (i.e., the optimal $\phi^*$ in \eqref{eq:OTdivergence:dual}) in the dual variational representation of proximal OT divergences.  Here we show that the variational derivative  is always well-defined for any probability measure
perturbations for which the proximal OT divergence remains finite.
We note that such perturbations can be singular and do not necessarily need to be absolutely continuous with respect to $Q$ where the variational derivative is calculated.
Therefore, variational derivatives of proximal OT divergences exist  in ``all directions," not just along absolutely continuous perturbations of the baseline $Q$ directions, which is the case, for instance, in KL and other $f$-divergences. A key practical outcome of this observation is that machine learning algorithms that involve the optimization of $f$-divergences of any type,  will be \emph{provably} stabilized when a proximal OT divergence is deployed instead of an $f$-divergence. For instance, in gradient-based algorithms used to solve the optimization problem $\min_{\theta}\mathfrak{D}_{\mathrm{KL},\varepsilon}^{c}(P^{\theta}\|Q)$, it is necessary to compute the gradient
\[
\nabla_\theta \mathfrak{D}_{\mathrm{KL},\varepsilon}^{c}(P^{\theta}\|Q)=\int\frac{\delta \mathfrak{D}_{\mathrm{KL},\varepsilon}^{c}(P\|Q)}{\delta P}(P^{\theta}(x))\nabla_\theta P^{\theta}(x)dx,
\]
which involves the variational derivative $\frac{\delta \mathfrak{D}_{\mathrm{KL},\varepsilon}^{c}(P\|Q)}{\delta P}$ which must be well-defined for the gradient to be properly evaluated. Moreover, the construction of gradient flows in the space of probability measures also relies critically on this variational derivative as discussed in detail in Section \ref{subsect: GF}.



\subsection{Proximal OT divergences have well-defined variational derivative }\label{sec:firstvariation}
From Theorem \ref{thm:duality.optimizers}, the optimizers $(\phi^*,\psi^*)$  are uniquely determined in $\textrm{supp}(P)$ and $\textrm{supp}(Q)$ respectively,  up to additive constant, i.e., that is we can replace $(\phi^*,\psi^*)$ by $( \phi^* + a,\psi^*-a)$ for any $a\in \mathbb{R}$.
To compute the directional derivative of the optimal transport divergence, we extend $\phi^*$ to all of $X$. This extension follows a standard approach repeatedly used in optimal transport theory, i.e.,  we define
\begin{equation}\label{eq:phistarextension}
\widehat{\phi}^*(x) = \inf_{y \in \textrm{supp}(Q)} \left\{ c(x,y)- \psi^*(y) \right\} \,.
\end{equation}
 Clearly $\widehat{\phi}^*(x)=\phi^*(x)$ for all $x \in \textrm{supp}(P)$ and by definition $\widehat{\phi}^*(x)$ is the largest
function $h(x)$ such that $h(x)=\phi^*(x)$ on $\textrm{supp}(P)$ and $h(x)+\psi^*(y)\le c(x,y)$ 
for $x\in X$ and $y \in \textrm{supp}(Q)$.  Our next theorem states that the variational derivative is equal to $\widehat{\phi}^*(x)$ and extends the results proved in \cite{dupuis2022formulation,Gu_GPA} for the proximal 1-Wasserstein divergence. For simplicity, we focus on the case where $c$ is bounded and uniformly continuous, though the result likely extends to more general cost functions as well.

\begin{theorem}{\bf (Variational derivative of the proximal OT-divergence)}\label{thm:first variation}  Suppose $c$ is uniformly continuous. 
Let  $\rho \in \mathcal{M}(X)$ be a finite signed measure with total mass $0$, $\rho(X)=0$, such that $\rho=\rho_+-\rho_-$ and $\rho_\pm \in \mathcal{P}(X)$ are mutually singular.  Assume that for some   $\alpha_0>0$,   $P+\alpha \rho \in \mathcal{P}(X)$ if $\alpha \le \alpha_0$. Then we have 
\begin{equation}
\lim_{\alpha \to 0} \frac{\mathfrak{D}_{\mathrm{KL},\varepsilon}^{c}(P + \alpha \rho\|Q) - \mathfrak{D}_{\mathrm{KL},\varepsilon}^{c}(P \|Q)}{\alpha}  =  \int \widehat{\phi}^{*} d\rho
\end{equation}
where $\widehat{\phi}^{*}$ is given by \eqref{eq:phistarextension}
and $(\phi^*,\psi^*)$ are the optimizer in 
\eqref{eq:OTdivergence:dual}.
\end{theorem}
The proof of the theorem is given in Appendix~\ref{sec:proofs:firstvariation}. In the sequel, we will use the following symbolic notation involving the variational derivative,  
\begin{equation}
\frac{\delta  \mathfrak{D}_{\mathrm{KL},\varepsilon}^{c}(P\|Q)}{\delta P}= \widehat{\phi}^{*}.
\end{equation}
This allows to express Theorem \ref{thm:first variation} in a concise form.
\begin{remark}
 From Theorem~\ref{thm:first variation} (applied to $\rho$ and $-\rho$), we conclude that if $P+\alpha \rho \in \mathcal{P}(X)$ for $\alpha$ in a neighborhood of $0$ then   $\mathfrak{D}_{\mathrm{KL},\varepsilon}^{c}(P + \alpha \rho\|Q)$ is actually differentiable at $0$. However, for differentiability to hold,  $\rho$ must be absolutely continuous with respect to $P$. In the general case, we only obtain a one-sided derivative, but this is precisely what is required for constructing a gradient descent algorithm in probability space. 

\end{remark}

\section{Algorithmic aspects of proximal OT divergences}
\label{sec:algorithms}

The computational aspects of proximal divergences can be broadly categorized into two interrelated challenges: (1) computing or estimating the proximal OT divergence itself using its primal and dual formulation, and (2) computing the transport proximal operator $\textbf{prox}^c_{\varepsilon \mathfrak{D}}$ given in Definition \ref{def:  Transport proximal operator}. Notably, there are computable formulations of the proximal OT divergence and its associated operators, for which various numerical methods have been proposed in the literature (see, e.g., \cite{onken2021ot,NEURALODE}).


\subsection{Proximal OT divergences for discrete distributions}
\label{sec:algorithms.discrete}
Proximal OT divergences are computationally tractable due to their dual and dynamic formulations, which offer two alternative algorithmic, as we discuss in the following sections. We first focus on computing the proximal OT divergence in the discrete setting, particularly when $P$ and $Q$ are empirical measures. In this case, the dual variational formulation of the proximal OT divergence directly translates into a convex optimization problem, allowing for efficient numerical computation. Specifically, let $P = \sum_{i=1}^N a_i \delta_{x^i},Q = \sum_{j=1}^M b_j \delta_{y^i}$ be two empirical measures where $(a_i)_{i=1}^N,(b_j)_{j=1}^M$ are weights in the probability simplex.  We write the intermediate measure as $R = \sum_{i=1}^N w_i\delta_{y_j}$  ensuring that $R$ remains absolutely continuous with respect to $Q$. The primal formulation of the proximal OT divergence with the KL divergence is given by 
\begin{equation}\label{eq:discrete_primal}
    \mathfrak{D}^{c}_{\varepsilon}(P\|Q) = \min_{(w_i)} \left\{\min_{\pi} \langle C,  \pi \rangle + \epsilon \sum_{i=1}^{M} w_i \log \frac{w_i}{b_i} \right\},
\end{equation}
where $C\in \R^{N \times M}$ is the matrix with entries $C_{i,j}=c(x^i,y^j)$ and $\pi \in \R^{N \times M}$ is a (discrete) transport plan from $P$ to $Q$, i.e.,  $\pi \in \Pi(a,c)$ where
\begin{equation}
\Pi(a,c) = \left\{\pi \in \R^{N \times M}, \; \pi \geq 0, \; \pi 1_N = a, \; \pi^T 1_M = c \right\}.
\end{equation}
Alternatively, the proximal OT-divergence can be written using the dual formulation in terms of potential $\phi \in \R^N, \psi\in \R^M$ as 
\begin{equation} \label{eq:discrete_dual}
    \mathfrak{D}^{c}_{\varepsilon}(P\|Q) = \max_{\phi\oplus\psi \le C} \left\{\sum_{i=1}^N \phi_i a_i - \epsilon \log \sum_{j=1}^M b_j e^{-\psi_j/\epsilon} \right\},
\end{equation}
where the constraint means that $\phi_i + \psi_j \leq C_{ij}$ for $i=1,\dots,N$ and $j = 1,\dots,M$. Moreover, the weights $w_i$ of $R$ are given by 
\begin{equation}
w_j = \frac{e^{-\psi_j/\varepsilon}}{\sum_{l}e^{-\psi_l/\varepsilon}}
\end{equation}
which comes from \eqref{eq:optimalR*2}. The problem in~\eqref{eq:discrete_dual} has a concave objective (by noting that the log-sum-exp function is convex) with linear inequality constraints. The problem can be solved using convex programming tools.

\subsection{Neural network approximations to proximal OT divergences}\label{sec:algorithms.CNN}
Let us consider the proximal OT divergence with the KL divergence and 2-Wasserstein distance, i.e., the  objective for the quadratic cost $c(x,y) = \frac{1}{2}|x-y|^2$. The dual representation \eqref{eq:OTdivergence:dual} for proximal OT divergences requires optimizing over the set $\Phi_c$ which poses significant computational challenges. To address this, we use the convexification trick introduced in \cite{Villani2003topics}. Building on the approach in \cite{makkuva2020optimal},  we reformulate the problem by enforcing convexity constraints on the function space. This reformulation allows us to express proximal OT divergences as a minimax optimization over convex functions, enabling the use of input convex neural networks (ICNNs) for efficient computation.

Specifically, we perform the change of variables $\widetilde\varphi(x) := \frac{1}{2}\|x\|^2 - \varphi(x)$ and $\widetilde\psi(y) := \frac{1}{2}\|y\|^2 - \psi(y)$ and then, the constraint in the dual problem can be written as $\widetilde\varphi(x) + \widetilde\psi(y) \geq xy$. Moreover, it is not too difficult to show that $\widetilde\varphi$ is convex and $\frac{1}{2}\|y\|^2 - \psi(y) = \frac{1}{2}\|y\|^2 - \varphi^c(y) = \widetilde\varphi^*(y)$ is the Legendre transform of $\widetilde\varphi$ at optimality, see Proposition 1.21 in~\cite{Santambrogio}. Then the optimization problem \eqref{eq:OTdivergence:dual} can be expressed as
\begin{eqnarray} \label{eq:variational_OTDiv}
&&\sup_{\widetilde\varphi} \left\{\mathbb{E}_P[-\widetilde\varphi] - \epsilon \log \mathbb{E}_Q[e^{\widetilde\varphi^*/\epsilon}e^{-\|y\|^2/2\epsilon}] \right\} + C \nonumber\\
&&\qquad= -\inf_{\widetilde\varphi} \left\{\mathbb{E}_P[\widetilde\varphi] + \epsilon \log \mathbb{E}_Q[e^{\widetilde\varphi^*/\epsilon}e^{-\|y\|^2/2\epsilon}] \right\} + C,
\end{eqnarray}
where $C = \frac{1}{2}\mathbb{E}_P\|x\|^2$ and $\widetilde\varphi$ is a convex and continuous function on $X$. This reformulation can also be interpreted as expressing the objective in terms of a modified cost function $c(x,y) = -xy$ and a perturbed measure $Q$ obtained by multiplying $Q$ by a standard Gaussian density with variance $\epsilon$ followed by re-normalization.

In practice, we solve this by parameterizing $\widetilde\varphi$ with an input-convex neural network in the class \texttt{ICNN}. From the proof of Theorem 3.3 in~\cite{makkuva2020optimal} we have that
$$\langle y, \nabla g(y)\rangle - \widetilde\varphi(\nabla g(y)) \leq \langle y, \nabla \widetilde\varphi^*(y)\rangle - \widetilde\varphi(\nabla \widetilde\varphi^*(y)) = \widetilde\varphi^*(y),$$
for all differentiable $g,\widetilde\varphi$. Moreover, the supremum of the left-hand-side over convex functions $g \in \texttt{CVX}$ achieves the Legendre transform $g=\widetilde\varphi^*$.
Given that $x \mapsto \epsilon\log \mathbb{E}_Q[e^{x/\epsilon}e^{-\|y\|^2/2\epsilon}]$ is an increasing function, we have that 
$$\sup_{g \in \texttt{CVX}}\epsilon\log \mathbb{E}_Q[e^{(\langle y, \nabla g(y)\rangle - \widetilde\varphi(\nabla g(y)))/\epsilon}e^{-\|y\|^2/2\epsilon}] \leq \epsilon\log \mathbb{E}_Q[e^{-\widetilde\varphi^*(y)}e^{-\|y\|^2/2\epsilon}].$$
Using this result, we can compute the OT divergence in~\eqref{eq:variational_OTDiv} by solving the variational problem
\begin{equation}\label{pOT:ICNN}\inf_{\widetilde\varphi \in \texttt{ICNN}} \sup_{g \in \texttt{ICNN}} \left\{\mathbb{E}_P[\widetilde\varphi] + \epsilon \log \mathbb{E}_Q[e^{(\langle y, \nabla g(y)\rangle - \widetilde\varphi(\nabla g(y)))/\epsilon}e^{-\|y\|^2/2\epsilon}] \right\}.
\end{equation}

\subsection{Generative flow-based algorithms}\label{sec:algorithms.MFG}


    Here, we demonstrate that a specific class of normalizing flows, known as the optimal-transport flow (OT-flow) introduced in \cite{onken2021ot}, is derived from a proximal OT divergence. The OT-flow is a type of generative model that was formulated as a stable version of the normalizing flows  \cite{rezende2015variational, tabak2010density, tabak2013family}. Given only samples from the distribution $Q$, a normalizing flow is used to learn a transformation that maps these samples to a simpler, typically known distribution, such as a standard Gaussian. The learned transformation is called the flow. The inverse of this flow serves as the generative model, which transforms samples from the simpler distribution $P$ into new samples that approximate the complex distribution $Q$. In \cite{onken2021ot}, the flow is found via maximum likelihood estimation, which is regularized by the Wasserstein distance and the connection with the proximal OT divergences is immediate:
     \begin{eqnarray}\label{opt:OT_flow}
        &&\min_{v,\rho} \Bigg\{\mathbb{E}_{\rho(x,T)}\left[\log \frac{\rho(x,T)}{P(x)} \right] + \int_0^T \int_{\R^d} \frac{1}{2} |v(x,t)|^2 \rho(x,t) dx \, dt\nonumber\\
        &&\hspace{2cm}:\frac{\partial \rho}{\partial t} + \nabla \cdot (v\rho) = 0,\, \rho(x,0) = Q(x)\Bigg\}\nonumber\\
        &&= 2T \inf_{R} \left[\mathfrak{D}_{\mathrm{KL}}(R\|P) + \frac{1}{2T} W_2^2(Q,R) \right]=\mathfrak{D}^2_{\mathrm{KL}, 2T}(Q\|P).
    \end{eqnarray}
    This framework allows us to compute our divergences and the proximal minimizer when KL and 2-Wasserstein efficiently by Neural ODEs \cite{NEURALODE} and OT-flow implemented by the tools outlined in Section 3 of \cite{onken2021ot}. 
    Training an OT-flow requires both forward and backward simulations of invertible flows, and can be unstable in high-dimensional examples where data often lies on a lower-dimensional manifold. In fact, the choice of KL in \eqref{opt:OT_flow} does not guarantee the existence of an optimizer due to the lack of absolute continuity, and there is no invertible mapping between  the reference distribution being often a $d$-dimensional Gaussian and the target distribution supported on a lower-dimensional manifold. $W_1\bigoplus W_2$ proximal generative flows in \cite{gu2024combining}, address these challenges, as they do not require the inversion of the flow or absolute continuity between the measures.


\subsection{Application to conditional sampling} \rb{Finally, we demonstrate how the divergence can be used to seek generative models that sample target conditional distributions. Let $x \in X$ and $z \in Z$ be two variables that follow a joint distribution $Q(x,z)$. For example, $x$ may be parameter in an inverse problem, endowed with a prior distribution $Q(x)$, and $z$ may be observation, which is related to $x$ by a stochastic forward model, i.e., $z = G(x) + \eta$ where $\eta$ is an (e.g., additive Gaussian) observational noise variable.}

\rb{Here we seek an approximation to $Q$ that factorizes as $Q(x,z) = Q(x|z)Q(z)$ in order to extract the conditional distribution for the parameter given the observation, which is known as the \emph{posterior distribution} in Bayesian inverse problems. The approximate distribution may be parameterized by a triangular transport map $T(x,z) = [T_x(x,z), T_z(z)]$ where $T_x : X \times Z \rightarrow X$ and $T_z : Z \rightarrow Z$ such that $T$  pushes forward any simple product distribution $P(x)P(z)$ to $Q(x,z)$. In particular, setting $P(z) = Q(z)$ and $T_z(z) = Id(z)$, then $T_x(\cdot,z)$ pushes forward $P(x)$ to $Q(x|z)$ for all $z$ and can be used for conditional sampling; see~\cite{baptista2024conditional} for more details. Similarly to the unconditional setting in Section~\ref{sec:algorithms.MFG}, we can seek the map $T_x$ as a flow via maximum likelihood estimation regularized by the Wasserstein-2 distance as in~\eqref{opt:OT_flow}. The variational problem for the flow can be written as: 
     \begin{eqnarray}\label{eq:cond_OT_flow}
        &&\min_{v,\rho} \Bigg\{\mathbb{E}_{\rho(x,z,T)}\left[\log \frac{\rho(x,z,T)}{P(x)P(z)} \right] + \int_0^T \int \int_{\R^d} \frac{1}{2} |v(x,z,t)|^2 \rho(x,z,t) dx dz \, dt\nonumber\\
        &&\hspace{2cm}:\frac{\partial \rho}{\partial t} + \nabla \cdot (v\rho) = 0,\, \rho(x,z,0) = Q(x,z)\Bigg\} = \mathfrak{D}^2_{\mathrm{KL}, 2T}(Q\|P)\,,
    \end{eqnarray}
where the velocity field $v : X \times Z \times [0,T] \rightarrow X$ is evolving the $x$ parameter for each observation $z$. We emphasize that the constraint on the flow to act only on $x$ in~\eqref{eq:cond_OT_flow}, results in a sequence of marginal distributions of the form $\rho(x,z,t) = \rho(x|z,t)Q(z)$ for all $z$; see~\cite{wang2023efficient} for an instantiation and application of the regularized divergence in~\eqref{eq:cond_OT_flow} for solving Bayesian inverse problems.}


\section{Gradient flows for proximal OT divergences and connections to Probability flows}\label{subsect: GF}


In this section, we consider  $Q$ as a fixed target measure --- typically arising from observed data --- and explore how the proximal OT divergence
$\mathfrak{D}^{c}_{\varepsilon}(P\|Q)$, can be leveraged to transport a source measure $P$ toward $Q$ by building gradient flows for $\mathfrak{D}^{c}_{\varepsilon}(P\|Q)$. This has been extensively studied in the recent work \cite{Gu_GPA} for the special case of the 1-Wasserstein proximal divergence under the assumption that $Q$ has a finite first moment, and   \cite{chen2024learning} where the authors demonstrate the robustness of the 1-Wasserstein $\alpha$-divergences in generative modeling without requiring any assumptions on $Q$. Finally, we  establish connections with probability flow ODEs. These ODEs provide a deterministic framework for evolving distributions over time and have recently emerged as powerful tools in generative modeling. 

A central step in constructing the gradient flow is the computation of the variational derivative of the divergence, denoted informally as $\frac{\delta  \mathfrak{D}^{c}_{\varepsilon}(P\|Q)}{\delta P}$ given in Theorem \ref{thm:first variation}, not only facilitates the definition of the flow but also sheds light on the regularizing influence of the entropic term. In particular, the entropic regularization ensures the existence and uniqueness of a well-defined variational derivative--something that is generally lacking in the classical optimal transport setting (see, e.g., Section 7.2 of \cite{Santambrogio}).  Conversely, adding the optimal transport term to the KL divergence alters its properties in a meaningful way, given that the KL divergence itself already has a well-defined variational derivative, i.e., 
$\frac{\delta  \mathfrak{D}_{\mathrm{KL}}(P\|Q)}{\delta P} = 1+\log\frac{dP}{dQ}$
provided that $P$ is absolutely continuous to $Q$. Regularizing with the optimal transport,  the variational derivative can be extended to apply to more general, non-absolutely continuous measures $P$.  Let $P$ be a source probability distribution and $Q$ an arbitrary target distribution. At the formal level, a gradient flow on the space of probability measures is obtained as the solution to the following partial differential equation
\begin{equation}
\begin{aligned}
     & \partial_{t}P_t ={\rm div}\left(P_t\nabla\frac{\delta  \mathfrak{D}^{c}_{\varepsilon}(P_t\|Q)}{\delta P_t}\right)  \,, \quad P_0=P\label{eq:transport:variational:pde}
     \end{aligned}
\end{equation} 
As noted in Remark \ref{rem: Connection to JKO scheme}, the gradient flow \eqref{eq:gradient flow JKO} in the JKO scheme shares the same structure as \eqref{eq:transport:variational:pde} and finds  the approximate solution at each time-step by solving \eqref{eq:iterative minimization}. Here, we aim to leverage the formulation in \eqref{eq:transport:variational:pde} to construct generative particle algorithms as discussed next. Recalling the variational derivative of the proximal OT-divergence with the KL divergence given in Theorem \ref{thm:first variation}, i.e.,
\[
\frac{\delta  \mathfrak{D}_{\mathrm{KL},\varepsilon}^{c}(P\|Q)}{\delta P}= \widehat{\phi}^{*} \quad\mathrm{with}\quad\widehat{\phi}^{*}(x)=\inf_{y \in \textrm{supp}(Q)} \left\{ c(x,y)- \psi^*(y) \right\}\,,
\]
and 
\[
(\phi^*,\psi^*)=\textrm{argmax}_{(\phi,\psi)\in\Phi_c}\left\{\mathbb{E}_{P}[\phi] - \varepsilon \log \mathbb{E}_Q[ e^{-\psi/\varepsilon}]\right\}
\]
we construct the gradient flows for proximal OT divergences as follows
\begin{align}\label{eq:fgdivergence:gradflow}
\partial_{t}P_t&={\rm div}\left(P_t\nabla \widehat{\phi}_t^{*} \right)\,,\quad P_0=P ,\nonumber\\
 \widehat{\phi}^*_t(x) &= \inf_{y \in \textrm{supp}(Q)} \left\{ c(x,y)- \psi^*_t(y) \right\}\\
(\phi_t^*,\psi_t^*)&=\textrm{argmax}_{(\phi,\psi)\in\Phi_c}\left\{\mathbb{E}_{P_t}[\phi] - \varepsilon \log \mathbb{E}_Q[ e^{-\psi/\varepsilon}]\right\}\nonumber
\end{align}
In the context of generative models, both the target distribution $Q$ and the generative model $P_t$ in \eqref{eq:transport:variational:pde}
are accessible only through samples and their corresponding empirical distributions. It has been mentioned that $\mathfrak{D}_{\mathrm{KL},\varepsilon}^{c}(P\|Q)$ can directly compare such singular distributions as ensured by the upper bound in \eqref{eq:upper.bound}. We now turn to the Lagrangian formulation of the PDE \eqref{eq:transport:variational:pde}, which corresponds to its underlying ODE/variational representation.
\begin{align}\label{eq:fgdivergenceODE}
\frac{d Y_t}{dt}&=-\nabla \widehat{\phi}_t^{*}(Y_t)\, ,\quad Y_0\sim P_0=P\nonumber\\
 \widehat{\phi}^*_t(x) &= \inf_{y \in \textrm{supp}(Q)} \left\{ c(x,y)- \psi^*_t(y) \right\}\\
(\phi_t^*,\psi_t^*)&=\textrm{argmax}_{(\phi,\psi)\in\Phi_c}\left\{\mathbb{E}_{P_t}[\phi] - \varepsilon \log \mathbb{E}_Q[ e^{-\psi/\varepsilon}]\right\}\nonumber
\end{align}
To construct a particle algorithm based on \eqref{eq:fgdivergenceODE}, we begin with the input data $(X^{(i)})_{i=1}^N$ from the unknown target distribution $Q$ and $(Y^{(j)})_{j=1}^M$ drawn from a source distribution  $P$. The corresponding empirical measures are $\hat{Q}^N$ and $\hat{P}_0^M$. The Euler time discretization of \eqref{eq:fgdivergenceODE} takes the form:
\begin{align}\label{eq:fgdivergenceDiscreteScheme}
Y_{n+1}^{(j)}&=Y_{n}^{(j)}-\Delta t\nabla \widehat{\phi}_n^{*}(Y_n^{(j)})\, ,\quad Y_0^{(j)}=Y^{(j)}\sim P_0,\qquad j=1,\dots, M\nonumber\\
\phi^*_n(Y_n^{(j)}) &= \min_{i=1,\dots,N} \left\{ c(Y_n^{(j)},X^{(i)})- \psi^*_n(X^{(i)}) \right\},\qquad j=1,\dots, M\nonumber \\
(\phi_n^*,\psi_n^*)&=\textrm{argmax}_{(\phi,\psi)\in\Phi_c}\left\{\frac{\sum_{j=1}^M\phi(Y_n^{(i)})}{N} - \varepsilon \log \frac{\sum_{i=1}^N e^{-\psi(X^{(i)})/\varepsilon}}{N}\right\}
\end{align}
where we remind that $\Phi_c = \left\{ (\phi, \psi) \in C_b(\mathbb{R}^d\times \mathbb{R}^d), \phi \oplus \psi \le c \right\}$. The discrete scheme \eqref{eq:fgdivergenceDiscreteScheme} can be interpreted as a generative algorithm that transports the initial source samples $(Y_0^{(j)})_{j=1}^M$ from $P_0$ (typically  a distribution that is easy to sample) over a time horizon $T=n_T\Delta t$,  where $n_T$ is the total number of steps, to  a new set of generated data $(Y_{n_T}^{(j)})_{j=1}^M$ that approximate samples from $Q$. 

 In \cite{Gu_GPA}, the particular case of $\mathfrak{D}_{\mathrm{KL},\varepsilon}^{1}$--the proximal OT divergence with 1-Wasserstein distance--is studied. Consequently, the optimization problem in \eqref{eq:fgdivergenceDiscreteScheme} is formulated over the space of Lipschitz functions. To make this tractable in practice, this function space $\mathrm{Lip_{\frac{1}{\varepsilon}}({\mathbb{R}}^d)}$ is approximated using neural networks $\Gamma^{NN}_{\frac{1}{\varepsilon}}$,  i.e., $\phi_n^*=\textrm{argmax}_{\phi\in \Gamma^{NN}_{\frac{1}{\varepsilon}}}\left\{\frac{\sum_{j=1}^M\phi(Y_n^{(i)})}{N} - \varepsilon \log \frac{\sum_{i=1}^N e^{\phi(X^{(i)})/\varepsilon}}{N}\right\}$
 where the Lipschitz constraint is enforced through spectral normalization techniques \cite{Miyato_spectral_norm}. Moreover, the optimization problem in the discrete scheme \eqref{eq:fgdivergenceDiscreteScheme} for the gradient flow of $\mathfrak{D}_{\mathrm{KL},\varepsilon}^2$ can be  implemented by solving the variational problem \eqref{pOT:ICNN} using input convex neural networks. {\color{black} The gradient particle scheme in \eqref{eq:fgdivergenceDiscreteScheme} can also be augmented with a force-matching step to recover a continuous spatiotemporal vector field that approximates \(-\nabla \widehat{\phi}_n^*\) throughout space and time, not just along particle trajectories. This extension enables interpolation, finer temporal resolution, and makes the method fully generative, allowing simulation from arbitrary initial data; see \cite{Cheng2025PROFET}. Force matching, conceptually related to score matching \eqref{eq:scorematching}, was originally introduced in molecular dynamics to reconstruct interatomic potentials from \emph{ab initio} force data \cite{ErcolessiAdams1994}, and has since found widespread application in coarse-graining of molecular simulations, see for example the recent review \cite{Noid2024}.}

\textbf{Connections to Probability Flows and score-matching:} Next, we explore the connection between gradient flows and probability flow ODEs (i.e., deterministic sampling), \cite{MaoutsaReichOpper}. 
Given an initial condition $p_0$, the Fokker-Planck equation 
\[
\partial_t p(t,y) = - {\rm{div}}(f(t,y) p(t,y)) + \frac{1}{2} g^2(t) \Delta p(t,y)
\]
describes the evolution of the marginal density $p_t$ for the random variable $Y_t$ following the forward SDE 
\[
dY_t=f(t,Y_t)dt+g(t) dW_t
\]
where $f(t,\cdot):\mathbb{R}^d\to\mathbb{R}^d$ and $g(t)\in\mathbb{R}$ are the drift and diffusion coefficient respectively, and $W_t$ is a standard Brownian motion. The Laplacian of $p$ can be written as a transport term driven by the score function $\nabla\log p$, i.e., 
\[
\Delta p={\rm{div}}(p\nabla\log p)
\]
and the Fokker-Planck equation is written as  
\begin{equation*}
\frac{\partial p_t(y)}{\partial t}=-{\rm div}\left(p_t (y)v(t,y)\right),\qquad v(t,y)=f(t, y) - \frac{g(t)^2}{2}  \underbrace{\nabla \log p_t(y)}_{\textrm{score function of $p_t$}}
\end{equation*}
The corresponding probability flow ODE is given by 
\begin{equation} \label{eq:probflowODE}
\frac{d Y_t}{dt}=v(t,Y_t),\quad Y_0\sim p_0
\end{equation}
This perspective  also forms the core idea in \cite{Song2021ScoreBasedGM}, which used as a deterministic alternative to the reverse-time stochastic differential equation (SDE), offering a new way to generate samples and compute exact likelihoods.  In fact, the authors first use the forward SDE to gradually add noise to the data, then train a time-dependent score-based model $s_\theta(y)$ by minimizing the Fisher divergence between the model and the data distributions defined as
\begin{equation}\label{eq:scorematching}
\mathbb{E}_{p(y)}\left[\|s_\theta(y)-\nabla \log p_t(y)\|^2_2\right]
\end{equation}
using score matching techniques (e.g., \cite{song2020sliced}), and finally generate samples by solving the corresponding probability flow ODE. \rb{Flow matching~\cite{lipmanflow} and stochastic interpolants~\cite{albergo2023building} are related frameworks that also seek a velocity field of the form in~(\ref{eq:probflowODE}) to match a prescribed path of marginal distributions across time.}

In our gradient flow perspective, starting from the simple case of the classical KL-divergence and computing its variational derivative, reveals a key connection to score-based modeling. Specifically, the gradient of the variational derivative of the KL divergence--equal to the gradient of the discriminator--can be written in terms of the score function: $\nabla\frac{\delta  \mathfrak{D}_{\mathrm{KL}}(p_t\|q)}{\delta p_t}=\nabla \phi^*_t=\nabla\log\frac{p_t}{q}=\nabla\log p_t-\nabla\log q$ and thus
\begin{eqnarray}\label{eq:score_vs_discriminators}
\textrm{score function of $p_t$ }=\nabla\log p_t=\nabla \phi^*_t+\nabla\log q
\end{eqnarray}
This formulation illustrates how score-based modeling  naturally arises from the  KL-divergence, thereby linking our gradient flow framework to the probability flow ODE. In fact, for specific choices of the functions $g(t)$ and $f(t,y)$, this yields the gradient flow of the KL divergence, which is further generalized by the gradient flow associated with proximal OT divergences in \eqref{eq:transport:variational:pde}. From a computational standpoint, rather than training a score-based model directly, we solve the variational problem defined by \eqref{eq:fgdivergenceDiscreteScheme} using convex optimization techniques, as detailed in Section \ref{sec:algorithms.discrete} and neural networks and discriminators in analogy to GANs, \cite{Gu_GPA}.

In fact, our methods are a type of probability flow --- in the sense that we solve deterministic dynamics for the particles --- but the velocities fields are informed by proximal OT divergences rather than the score \cite{Song2021ScoreBasedGM} or the KL divergence \cite{MaoutsaReichOpper}.  In general, the probability flow corresponding to the gradient flow of any proximal OT divergence can be computed through approximating discriminators $\phi_t^\ast$ without explicitly knowing the distribution $q$. The resulting probability flow does recover the Wasserstein gradient flow corresponding to the usual KL divergence, i.e., \eqref{eq:score_vs_discriminators} reduces to the usual score function.

\bibliographystyle{plain}
\bibliography{references.bib}

\appendix 

\section{Proofs of the divergence theorems  in Sections \ref{sec:good.divergence}} 
\label{sec:proofs.good.divergence}

We start with Theorem \ref{thm:proximal} which is proved by 
using the lower semicontinuity of the KL-divergence and transport cost together with the compactness of the level sets of the KL-divergence.

\noindent
{\it Proof of Theorem \ref{thm:proximal}}.  The upper bound 
\eqref{eq:upper.bound} is easily obtained by picking $R=P$ and $R=Q$ and using the divergence property for $T_c(P,Q)$ and $\mathfrak{}P\|Q)$. 

To show the existence of the minimizer, let $R_n$ be a minimizing sequence such that 
\begin{equation}
T_c(P,R_n) + \epsilon \mathfrak{D}(R_n \|Q)
\le \mathfrak{D}^{c}_{\varepsilon}(P\|Q) + \frac{1}{n}
\end{equation}
Since the level sets $\mathfrak{}\cdot\|Q)$ are precompact in the weak topology there exists a convergent subsequence (which we denote again by $R_n$) such that $R_n \to R^*$ weakly. Since both $T_c(P,\cdot)$ and $\mathfrak{D}(\cdot \| Q)$  are lower semicontinuous in the weak topology we have 
\begin{equation}
\begin{aligned}
T_c(P,R^*) + \varepsilon \mathfrak{D}(R^* \|Q) &\le \liminf_{n \to \infty} \left\{T_c(P,R^n) + \varepsilon \mathfrak{D}(R^n \|Q) \right\}
\\
& \le \mathfrak{D}^{c}_{\varepsilon}(P\|Q) 
\end{aligned}
\end{equation}
which shows that $R^*$ is a minimizer.

The uniqueness follows from the strict convexity of $\mathfrak{}\cdot\|Q)$ and the convexity of  $T_c(P,\cdot)$. 
If $R_1^*$ and $R_2^*$ are distinct minimizers then set $R_3^*=\frac12(R_1^*+R_2^*)$.
We have then 
\begin{equation}
\begin{aligned}
&T_c(P,R_3^*) + \varepsilon \mathfrak{D}(R_3^* \|Q) \\
& <  
\frac{1}{2} \left( T_c(P,R_1^*) + \epsilon \mathfrak{D}(R_1^* \|Q) + T_c(P,R_2^*)  + \epsilon \mathfrak{D}(R_2^* \|Q) \right)\\
&= \mathfrak{D}^{c}_{\varepsilon}(P\|Q)\,.
\end{aligned}
\end{equation}
This contradicts the fact that 
$R_1^*$ and $R_2^*$ are minimizers. \qed

\bigskip

\noindent
{\it Proof of Theorem \ref{thm:divergence}}.The divergence property, convexity and lower semiconitnuity are proved as follows:

\noindent{\bf Divergence property:}
For the divergence property  the nonnegativity in  \eqref{eq:divergence} follows from the nonnegativity of  
$T_c(P,R)$ and $\mathfrak{D}(R\|Q)$.   
If the proximal OT-divergence $\mathfrak{D}^{c}_{\varepsilon}(P\|Q)=0$ then using the minimizer $R^*$ provided by Theorem \ref{thm:proximal} we have
\begin{equation}
0 = T_c(P,R^*) + \mathfrak{D}(R^*\|Q)
\end{equation}
and thus $R^*=Q=P$ and thus $P=Q$.

\medskip

\noindent{\bf Convexity:} Consider two pairs or probability measures $P_1,Q_1$ and $P_2,Q_2$ and the corresponding minimizers $R_1^*$ and $R_2^*$.  Note also that the function 
$F(P,Q,R) =T_C(P,R)+\varepsilon \mathfrak{D}(R\|Q)$ is jointly convex. We have then 
\begin{equation}
\begin{aligned}
&\mathfrak{D}^{c}_{\varepsilon}(\alpha P_1 + (1-\alpha)P_2\|\alpha Q_1 +(1-\alpha Q_2) \\
&\le F(\alpha(P_1,Q_1,R_1^*) + (1-\alpha)(P_2,Q_2,R_2^*))\\
& \le \alpha F(P_1,Q_1,R_1^*)   + (1-\alpha)F(P_2,Q_2,R_2^*)
 \\
&= \alpha \mathfrak{D}^{c}_{\varepsilon}(P_1 \| Q_1)  + (1-\alpha)\mathfrak{D}^{c}_{\varepsilon}(P_2\| Q_2)\,
\end{aligned}
\end{equation}
which proves convexity.
\medskip

\noindent {\bf Lower semicontinuity:}
For the lower semicontinuity property
let $(P_n,Q_n) \in  \mathcal{P}(X)\times \mathcal{P}(Y)$  
a weakly convergent sequence and suppose $n_k$ is such that 
\begin{equation}
\liminf_{n} \mathfrak{D}^{c}_{\varepsilon}(P_n\|Q_n) = \lim_{k} \mathfrak{D}^{c}_{\varepsilon}(P_{n_k}\|Q_{n_k})\,. 
\end{equation}
If the limit is infnite there is nothing to prove so we can assume the limit is finite. 
Denoting by $R_{n_k}^*$ the minimizer,  the sequence 
\begin{equation}
\mathfrak{D}^{c}_{\varepsilon}(P_{n_k}\|Q_{n_k})= T_c(P_{n_k},R_{n_k}^*) + \varepsilon \mathfrak{D}(R_{n_k}^*\|Q_{n_k}) 
\end{equation}
is bounded and therefore the sequence $R_{n_k}^*$ has a convergent subsequence
(also denoted by $R_{n_k}^*$) with limit $R$. 
Therefore, using the lower semicontinuity of $T_c$ and $\mathfrak{D}$ in both argument we have
\begin{equation}
\begin{aligned}
\liminf_{n} \mathfrak{D}^{c}_{\varepsilon}(P_n\|Q_n) &= \lim_{k} \mathfrak{D}^{c}_{\varepsilon}(P_{n_k}\|Q_{n_k})\\
&= \lim_{k} T_c(P_{n_k},R_{n_k}^*) + \varepsilon \mathfrak{D}(R_{n_k}^*\|Q_{n_k})  \\ 
&\ge T_c(P,R) + \varepsilon \mathfrak{D}(R\|Q)\\
&\ge  \mathfrak{D}^{c}_{\varepsilon}(P\|Q)
\end{aligned}
\end{equation}
which proves lower semicontinuity and concludes the the proof. \qed

\section{Proof of the interpolation Theorem 
in Section \ref{sec:interpolation}}
\label{sec:proofs.interpolation}

\noindent
{\it Proof of Theorem \ref{thm:interpolation}} The proof of the limits $\varepsilon \to 0$ and $\varepsilon \to \infty$ go as follows:
\medskip

\noindent {\bf  The limit $\varepsilon \to 0$.} Assume first that $\mathfrak{}P\|Q) <\infty$.
Then we have 
\begin{equation}\label{eq:upper0}
\frac{1}{\varepsilon} \mathfrak{D}^{c}_{\varepsilon}(P\|Q) = \inf_{R \in \mathcal{P}(Y)} \left\{ \mathfrak{D}(R \|Q) + \frac{1}{\varepsilon} T_c(P,R)\right\} \le 
\mathfrak{D}(P\|Q) < \infty \,.
\end{equation}
This implies that 
\begin{equation}
\limsup_{\varepsilon \to 0}
\frac{1}{\varepsilon} \mathfrak{D}^{c}_{\varepsilon}(P\|Q) \le \mathfrak{D}(P\|Q)\,.
\end{equation}
Suppose $\epsilon_n \to 0$ and let $R_n^*$ 
be the corresponding minimizers in \eqref{eq:upper0}. The bound on the KL divergence and \eqref{eq:upper0} implies that the sequence $\{R^*_n\}$ is weakly compact and so passing to a subsequence we can assume that $R^*_n \to R$. By the lower semicontinuity of $T_c$ and \eqref{eq:upper0} we have 
\begin{equation}\label{eq:e01}
T_c(P,R) \le \liminf_{n}\varepsilon_n \frac{1}{\varepsilon_n}T_c(P,R_n^*)  \le \lim_{n} \varepsilon_n \mathfrak{D}(P\|Q) =0
\end{equation}
and therefore $R^*=P$.  Finally, by the lower semicontinuity of $\mathfrak{D}(P\|Q)$
\begin{equation}\label{eq:e02}
\begin{aligned}
\liminf_{n} \frac{1}{\varepsilon_n} \mathfrak{D}^{c}_{\varepsilon_n}(P\|Q)& =
\liminf_{n}\left\{  \mathfrak{D}(R_n^* \|Q) + \frac{1}{\varepsilon_n} T_c(P, R_n^*)\right\} \\
& \ge \liminf_{n} \mathfrak{D}(R_n^* \|Q) \ge \mathfrak{D}(P\|Q)\,.
\end{aligned}
\end{equation}
Since the the sequence $\epsilon_n$ is arbitrary this concludes the proof in the case where $\mathfrak{D}(P\|Q)< \infty$.

Next let us consider the case where  $\mathfrak{D}(P\|Q) =\infty$ and so we need to show that $\liminf_{\varepsilon \to 0} \frac{1}{\varepsilon} \mathfrak{D}^{c}_{\varepsilon}(P\|Q) = \infty$.  By contradiction assume it does not and thus there exists a sequence $\varepsilon_n \to 0$ such that $\lim_{n} \frac{1}{\varepsilon_n} \mathfrak{D}^{c}_{\varepsilon_n}(P\|Q) =D< \infty$. For this sequence we can repeat the argument used in the case of finite $\mathfrak{D}(P\|Q) =\infty$ (see equations \eqref{eq:e01} and \eqref{eq:e02})
to show that, passing to a subsequence if necessary we have 
$\lim_{n} \frac{1}{\varepsilon_n} \mathfrak{D}_{c}^{\varepsilon_n}(P\|Q) \ge \mathfrak{D}(P\|Q)=+\infty$ which is a contradiction. 
\medskip

\noindent{\bf The limit $\epsilon \to \infty$}.
Assume first that $T_c(P,Q)<\infty$.  The argument is very similar to the limit $\varepsilon \to 0$.
We have 
\begin{equation}\label{eq:upper2}
 \mathfrak{D}_{\varepsilon}^c(P\|Q) \le \inf_{R \in \mathcal{P}(Y)} \left\{ T_c(P,R) + \varepsilon \mathfrak{D}(R \|Q) \right\} \le 
T_c(P,Q) < \infty \,.
\end{equation}
This implies immediately that 
$
\limsup_{\varepsilon \to \infty}
\mathfrak{D}^{c}_{\varepsilon}(P\|Q) \le T_c(P,Q)
$.
Arguing as above given any sequence $\varepsilon_n \to \infty$ the minimizers $R_n^*$
in \eqref{eq:upper2} have a convergent subsequence $R_n^* \to R$.  The lower semicontinuity of $\mathfrak{D}(\cdot\|Q)$ implies 
that 
\begin{equation}\label{eq:einfty1}
\mathfrak{D}(R\|Q) \le \liminf_{n}\frac{1}{\varepsilon_n} \varepsilon_n \mathfrak{D}(R_n^*\|Q)) \le \lim_{n} \frac{1}{\varepsilon_n} T_c(P,Q) =0.
\end{equation}
and thus $R=Q$. This implies that 
\begin{equation}\label{eq:einfty2}
\begin{aligned}
\liminf_{n}  \mathfrak{D}^{c}_{\varepsilon_n}(P\|Q)& =
\liminf_{n}\left\{ T_c(P, R_n^*) +\varepsilon_n   \mathfrak{D}(R_n^* \|Q) \right\} \\
& \ge \liminf_{n}  T_c(P, R_n^*) \ge T_c(P,Q)
\end{aligned}
\end{equation}
Since the the sequence $\epsilon_n$ is arbitrary this concludes the proof.  The case $T_c(P,R)=\infty$ is argued similarly as above this concludes the proof. \qed

\section{Proofs of the duality results 
in Section \ref{sec:duality}}
\label{sec:proofs.duality}

In this section we prove Theorem \ref{thm:duality.Polish}.  We start by proving this Theorem in the special case where $X$ and $Y$ are compact and $c$ is continuous, and we prove it by using similar convex optimization tools as the ones in \cite{Santambrogio} (see Theorem 1.46). 

We start by showing that  problem has a maximizer.  The proof is a slight generalization of the corresponding result in optimal transport and use the notion of $c$-transform which we now recall.  For $\phi:X\to\bar{\mathbb{R}}$, we define
\begin{eqnarray}
\phi^c(y) &=& \inf_{x} c(x,y) - \phi(x) 
\end{eqnarray}
 and for $\psi:Y\to\bar{\mathbb{R}}$, we define
\begin{eqnarray}
\psi^{\overline{c}}(x)&=&\inf_{y} c(x,y) - \psi(y)
\end{eqnarray}
We call a function $\psi(y)$ $c$-concave if $\psi=\phi^c$ for some function $\phi(x)$ (similarly for $\overline{c}$-concavity). The key property of $c$-concave functions on compact sets is that they form an equicontinuous family (see e.g \cite{Santambrogio}, p.12 and the argument below).

\begin{lemma}\label{lem:maxDP}
If $X$ and $Y$ are compact and $c$ is continuous then there exists a solution to the problem 
\begin{equation}\label{eq:dualoptmizer}
\sup_{(\phi, \psi) \in \Phi_c} \left\{ \mathbb{E}_P[\phi] - \varepsilon \ln \mathbb{E}_Q[ e^{-\psi/\varepsilon}] \right\}
\end{equation}
and $\varphi$ is c-concave, $\psi$ is $\overline{c}$-concave and $\psi=\varphi^c$. 
\end{lemma}

\proof  Consider a pair $(\phi, \psi)$ with $\phi\oplus\psi \le c$, that is $\psi(y) \le c(x,y)-\phi(x)$. Since the map 
$$
\psi  \mapsto  - \varepsilon \ln \mathbb{E}_Q[ e^{-\psi/\varepsilon}] 
$$
is increasing, the optimized functional in \eqref{eq:dualoptmizer}
always increases if $\psi$ increases and so we can always replace in the optimization  a pair $(\varphi, \psi)$ by $(\varphi, \varphi^c)$.  So we can always assume that $\psi$ is $c$-concave.  Similarly we can always replace a pair $(\varphi, \psi)$ by $(\psi^{\overline{c}}, \psi)$ and assume that $\phi$ is $\overline{c}$-concave.

Given a maximizing sequence $(\phi_n, \psi_n)$ in \eqref{eq:dualoptmizer} and applying $c$- and $\overline{c}$-transforms we obtain another maximizing sequence (for which we use the same symbols). Since $X$ and $Y$ are compact the function $c$ is uniformly continuous we denote by $\omega$ be its modulus on continuity (that is  $|c(x,y)-c(x',y')|\le \omega(d(x,x')+d(y,y'))$ for all $x,x'y,y'$).

By definition of the $c$-transform any $c$-concave function $\phi^c$ has
the same modulus of continuity as $c$. This implies that the family $(\phi_n, \psi_n)$ is equicontinuous. 

To get equiboundedness we note that replacing $(\phi_n,\psi_n)$
by $(\phi_n+a_n, \psi_n-a_n)$ for an arbitrary constant $a_n$ does not change the value of the functional in \eqref{eq:dualoptmizer}. By choosing the constants $a_n$ appropriately we can assume that $\phi_n(x)\ge 0$ and thus, by equicontinuity,  we can assume that $0 \le \phi_n \le \omega({\rm diam(X)})$ and since $\psi_n=\phi_n^c$ we get $\psi_n \in [\min(c)- \omega({\rm diam(X)}), \max(c)]$.
This shows that the sequence $(\phi_n,\psi_n)$ can be chosen to be equibounded. By Arzela-Ascoli Theorem there is a convergent subsequence converging uniformly to the  limit $(\phi,\psi)$ which  still satisfies $\phi \oplus \psi\le c$.  By uniform continuity we have 
\begin{equation}
\mathbb{E}_P[\phi_n] -\varepsilon \log \mathbb{E}_Q[e^{-\psi_n/\varepsilon}] 
\rightarrow \mathbb{E}_P[\phi] -\varepsilon \log \mathbb{E}_Q[e^{-\psi/\varepsilon}]
\end{equation}
and thus the supremum is attained in \eqref{eq:dualoptmizer}.
\qed

Let us consider the following map $F: C(X \times Y) \to \mathbb{R}$ given by 
\begin{equation}\label{eq:Fmap.definition}
F(p) = - \max_{(\phi, \psi)\in \Phi_{c-p}} \left\{ \mathbb{E}_P[\varphi] - \varepsilon \ln \mathbb{E}_Q[e^{-\psi/\varepsilon}]\right\}
\end{equation}
denotes a certain class of functions. By Lemma \ref{lem:maxDP}, the map $F$ is well-defined, and in particular, $F(0)$ corresponds to the negative of the right-hand side of \eqref{eq:OTdivergence:dual}.

To establish duality, we will employ convex duality, specifically showing that $F=F^{**}$ (the double Legendre transform). To proceed, we first need the following lemma.

\begin{lemma}\label{lem:Fmap.lsc.convex}
Suppose $X$ and $Y$ are compact. The map $F: C(X \times Y)\to \mathbb{R}$ is convex and lower semi-continuous
and thus by Legendre-Fenchel Theorem we have $F=F^{**}$ 
\end{lemma}

\begin{proof} For the convexity let  $p_0, p_1 \in C(X\times y)$ let $(\varphi_0, \psi_0)$ and $(\varphi_1, \psi_1)$ the corresponding maximizers (given by Lemma \ref{lem:maxDP}. Let $p_t = (1-t)p_0 + t p_1$ and  set $\varphi_t = (1-t) \varphi_0 + t \varphi_1$ 
and $\psi_t = (1-t) \psi_0 + t \psi_t$. 

The map $(\phi, \psi) \to \mathbb{E}_P[\varphi] - \varepsilon \ln \mathbb{E}_Q[e^{-\psi/\varepsilon}]$ is concave and thus 
\begin{equation}
\mathbb{E}_P[\varphi_t] - \varepsilon \ln \mathbb{E}_Q[e^{-\psi_t/\varepsilon}] \ge -(1-t) F(p_0) -t F(p_1)
\end{equation}
and thus $F$ is convex. 

For the lower semicontinuity suppose $p_k \to p$. Passing to a subsequence if necessary (which we denote by $p_k$ again) we can assume that $\liminf F(p_k)=\lim_{k} F(p_{k})$.  By the converse statement in Arzela-Ascoli theorem the sequence $\{p_{k}\}$ is equibounded and equicontinuous.   We now use the Lemma \ref{lem:maxDP} with cost $c-p_k$ and  obtain maximizers   $(\varphi_{k},\psi_{k})$ which are $c-p_{k}$ and $\overline{c-p_{k}}$ concave.  This implies that $(\varphi_{n_k},\psi_{n_k})$are also equicontinuous and equibounded (using the same arguments as in the proof of 
Lemma \ref{lem:maxDP}. 
By Arzela-Ascoli again, passing to another subsequence, we can assume that $\varphi_{k} \to \varphi$ and $\psi_{k}\to \psi$ uniformly. 
Since $\varphi_{k} \oplus\psi_{k}\le c-p_k$ we have then $\varphi \oplus \psi \le c-p$. So 
\begin{equation}
\begin{aligned}
\liminf F(p_n) &= \lim_{k} F(p_{k})  = - \lim_{k}  \left(\mathbb{E}_P[\varphi_{k}] - \varepsilon \ln \mathbb{E}_Q[e^{-\psi_{k}/\varepsilon}]\right) \\
&=  - \left(\mathbb{E}_P[\varphi] - \varepsilon \ln \mathbb{E}_Q[e^{-\psi/\varepsilon}]\right) \ge F(p)
\end{aligned}
\end{equation}
and thus $F$ is lower semicontinuous. 
\end{proof}

With these two technical lemmas in place, the core argument establishing duality is contained in the following proposition.

\begin{proposition}\label{prop:LTcomputation}
For the map $F(p)$ given in \eqref{eq:Fmap.definition} we have 
\begin{equation}
F^{**}(0) = - \inf_{R \in \mathcal{P}(Y)}\{ W(P,R) +\mathfrak{}R, Q)\}
\end{equation}
\end{proposition}

\begin{proof} 
Let us first compute the Legendre-Fenchel  transform of $F^*(\pi)$ of $F$.  
For $\pi \in \mathcal{M}(X \times Y)$ we have 
\begin{equation}
\begin{aligned}
F^*(\pi) & = \sup_{p\in C(X \times Y)} \left\{\int p d\pi - F(p)\right\}  \\
&= \sup_{p \in C(X \times Y)} \left\{ 
\int p d\pi  + \sup_{\varphi \oplus \psi\le c-p} 
\{ 
\mathbb{E}_P[\varphi] - \varepsilon \ln \mathbb{E}_Q[e^{-\psi/\varepsilon}]
\} 
\right\} 
\end{aligned}
\end{equation}
which can be written as a unique sup over $p, \phi, \psi$.  If $\pi \in \mathcal{M}(X \times Y)$ is not a positive Borel measure then these exist $q\le 0$ such that $\int q d\pi >0 $.   If we take then $\varphi=0$, $\psi=0$ and $p = c + nq$ then we find 
$F^*(\pi) \ge \int p d\pi + n \int q d\pi$ and thus $F^*(\pi)=+\infty$.  

If $\pi\in \mathcal{M}(X \times Y)$ is a positive Borel measure and $(\phi,\psi)$ is fixed to take the supremum over $p$  we should take $p$ 
as large as possible, that is $\phi \oplus \psi = c-p$ and thus 
\begin{equation}
\begin{aligned}
&F^*(\pi) = \sup_{\varphi, \psi} 
\left\{ 
\int (c - \varphi \oplus \psi) d\pi   + \mathbb{E}_P[\varphi] -\varepsilon\log \mathbb{E}_Q[e^{-\psi/\varepsilon}]
\right\}
\\
& = \int_{X \times Y} c  d\pi  + \sup_{\varphi} \left\{ \mathbb{E}_P[\varphi] - \int_{X \times Y} \varphi d\pi\right\} 
+ \varepsilon \sup_{\psi} \left\{ \int (-\psi/\varepsilon) d\pi 
- \ln \mathbb{E}_Q[ E^{-\psi/\varepsilon} ]
\right\} 
\end{aligned}
\end{equation}
Denoting by $\pi_X$ and $\pi_Y$ the marginals of $\pi$, we have for the supremum over $\varphi$
\begin{equation}
\sup_{\varphi} \left\{ \mathbb{E}_P[\varphi] - \int_{X \times Y} \varphi d\pi\right\}=\left\{
\begin{array}{cl} 0 & \textrm{if } \pi_X= P \\
\infty & \textrm{ otherwise} \end{array} 
\right.
\end{equation}
For the supremum over $\psi$, the Donsker-Varadhan representation of the KL-divergence gives
\begin{equation}
\sup_{\psi} \left\{ \int (-\psi/\varepsilon) d\pi 
- \ln \mathbb{E}_Q[ E^{-\psi/\varepsilon} ] \right\}= \left\{ 
\begin{array}{cl}  \mathfrak{}\pi_Y\|Q) & \textrm{if }\pi_Y \in \mathcal{P}(Y) \\
+ \infty & \textrm{if }  \pi_Y \in \mathcal{M}(Y) \setminus \mathcal{P}(Y) 
\end{array}
\right. .
\end{equation}
Therefore we find 
\begin{equation}
F^*(\pi) =  \left\{
\begin{array}{cl}
\mathbb{E}_{\pi}[c] + \varepsilon \mathfrak{} \pi_Y \|Q)  & \textrm{if } \pi_X =P \textrm{ and } \pi_Y \in \mathcal{P}(Y) \\
+\infty & \textrm{otherwise}
\end{array}
\right.
\end{equation}
Therefore 
\begin{equation}
\begin{aligned}
F^{**}(0) & =\sup_{\pi} \int 0 d\pi - F^*(\pi) \\
& =-  \inf_{R \in \mathcal{P}(Y)} \inf_{\pi \in \Pi(P,R)} \{ \mathbb{E}_{\pi}[c] + \varepsilon \mathfrak{} \pi_Y \|Q) \} \\
&= - \inf_{R \in \mathcal{P}(Y)}\{ W(P,R) + \varepsilon \mathfrak{D}(R, Q)\}
\end{aligned}
\end{equation}
and this concludes the proof. 
\end{proof}

\begin{proof}[Proof of Theorem \ref{thm:duality.optimizers}]
Let $R^*$ be the (unique) minimizer in \eqref{eq:def:OTdivergence}, 
$\pi^*$ an optimal transport plan between $P$ and $R^*$ and $(\phi^*,\psi^*)$ a maximizer for the \eqref{eq:dualoptmizer}.
We have then 
\begin{equation}
 \mathbb{E}_{\pi^*}[c] + \varepsilon \mathfrak{D}(R^*\|Q) = \mathbb{E}_P[\phi^*] -\varepsilon \ln \mathbb{E}_Q[e^{-\psi^*/\varepsilon}]
\end{equation}
which we can rewrite as 
\begin{equation}
\underbrace{ \mathbb{E}_{\pi^*}[c - \phi^* - \psi^*]
 }_{\ge 0}+ \varepsilon 
 \underbrace{\left\{ \mathfrak{D}(R^*\|Q) -  \mathbb{E}_{R}[-\psi^*/\varepsilon] + \log \mathbb{E}_Q[e^{-\psi^*/\varepsilon}] \right\}
 }_{\ge 0} = 0
\end{equation}
The first term is non-negative by definition of the transport cost and the second by Donsker-Varadhan's formula.  For the second term to be zero both equations \eqref{eq:optimalR*1} 
and \eqref{eq:optimalR*2} to hold and this concludes the proof.  We note further that $\psi^*$ is uniquely determined $Q$ almost surely (up to a constant) and thus, by continuity,  uniquely determined in $\textrm{supp}(Q)$.
From the
the first term in \eqref{eq:proofduality} we see that $\phi\oplus \psi=c$ $\pi^*$ almost surely and thus $\phi^*$ is uniquely determined $P$ almost surely (up to a constant) and also in $\textrm{supp} (P)$. 
This concludes the proof of Theorem \ref{thm:duality.optimizers}. 
\end{proof}

To extend the duality formula from the compact case to the general case we proceed in two steps.  We first relax the compactness assumption on $X$ and $Y$ but  assume that the cost $c$ is bounded and continuous.    

\begin{theorem}\label{thm:duality.bounded} Assume $X$ and $Y$ are Polish space and $c$ is bounded continuous
then we have 
\begin{equation}
\mathfrak{D}^c_\varepsilon(P\|Q) = \sup_{(\phi,\psi)\in \Phi_c} \left\{ \mathbb{E}_P[\phi] - \varepsilon\ln \mathbb{E}_Q[e^{-\psi/\varepsilon}] \right\}
\end{equation}
\end{theorem}

\begin{proof}  Recall that by duality for optimal transport and the KL-divergence we 
have for any $R \in \mathcal{P}(Y)$
\begin{equation}
\begin{aligned}
\mathbb{E}_P[\phi] -\varepsilon \ln \mathbb{E}_Q[e^{-\psi/\varepsilon}] &=   \mathbb{E}_P[\phi] +\mathbb{E}_R[\psi] + \varepsilon \left(\mathbb{E}_R[-\psi/\varepsilon] -\ln \mathbb{E}_Q[e^{-\psi/\varepsilon}]\right) \\
&\le T_c(P,R) + \mathfrak{D}(R\|Q)\\
\end{aligned}
\end{equation}
and thus 
\begin{equation}
\mathfrak{D}^c_\varepsilon(P\|Q) \ge  \sup_{(\phi,\psi)\in \Phi_c} \left\{ \mathbb{E}_P[\phi] - \varepsilon\ln \mathbb{E}_Q[e^{-\psi/\varepsilon}] \right\}.
\end{equation}

We show the reverse inequality. Since $X$ and $Y$ are Polish spaces, we pick a sequence $\delta_n \to 0$ and compact sets $X_n,Y_n$ such that $P(X\setminus X_n) \le \delta_n$ and $Q(Y\setminus Y_n) \le \delta_n$.  We denote by $P_n$ and $Q_n$ the restriction of $P$ to $X_n$ (i.e., $P_n(A)=P(A \cap X_n)/P(X_n)$). It is easy to verify that $P_n$
and $Q_n$  converge weakly to $P$ and $Q$ respectively. We also denote $c_n$ the restriction of $c$ to $X_n\times Y_n$ 
and we have then 
\begin{equation}
\mathfrak{D}^c_\varepsilon(P_n\|Q_n) = \inf_{R \in \mathcal{P}(Y_n)} \left\{ T_{c_n}(P_n,R) + \varepsilon \mathfrak{D}(R\|Q_n)\right\} = \mathfrak{D}^{c_n}_\varepsilon(P_n\|Q_n)
\end{equation}
and we can use the results obtained in the compact case.

Using the lower semicontinuity of  $\mathfrak{D}^c_\varepsilon(P\|Q)$ given an arbitrary $\alpha >0$ we can pick $N$ such that 
\begin{equation}\label{eq:tt23}
\begin{aligned}
\mathfrak{D}^c_\varepsilon(P\|Q)  &\le \mathfrak{D}^{c_N}_\varepsilon(P_N\|Q_N) + \alpha  \\
& = \sup_{(\phi,\psi)\le \Phi_{c_N}}
\left\{ \mathbb{E}_{P_N}[\phi] - \varepsilon\ln \mathbb{E}_{Q_N}[e^{-\psi/\varepsilon}] \right\} + \alpha\\
& = \mathbb{E}_{P_N}[\phi_N^*] - \varepsilon\ln \mathbb{E}_{Q_N}[e^{-\psi_N^*/\varepsilon}] + \alpha
\end{aligned}
\end{equation}
where we have used the duality for the compact case  and the existence of a maximizers (Lemma \ref{lem:maxDP}).  Recall that $\phi_N^*(x)+\psi_N^*(y)=c(x,y)$ is $x \in \textrm{supp}(P_N)$ and $y \in \textrm{supp}(Q_N)$.

Next (using the standard $c$-transform trick in optimal transport) we extend the pair $(\phi_N^*(x),\psi_N^*(y))$ to  $X$ and $Y$ by setting 
\begin{equation}\label{eq:phiNext}
\begin{aligned}
\widetilde{\phi}(x) &= \inf_{y \in X_N} \left\{ c(x,y) - \psi_N^*(y) \right\} \\
\widetilde{\psi}(y) &=\inf_{y \in Y} \left\{ c(x,y) - \widetilde{\phi}(y) \right\}
\end{aligned}
\end{equation}
This definition ensures that $\phi(x)=\phi_N(x)$ on $\textrm{supp}(P_N)$
and $\psi(x)=\psi_N(y)$ on $\textrm{supp}(Q_N)$ and moreover $\phi(x) + \psi(y)\le c(x,y)$ for all $x,y$.  

The fact that $c$ is bounded provides a bound on $\phi$ and $\psi$.  Without loss of generality we can assume that $c\ge 0$
and we pick $x_0,y_0$ such that $\phi_N(x_0)+\psi_N(x_0)=c(x_0,y_0)\ge 0$.  The pair $\phi_N,\psi_N$ is determined up to a constant which we can choose so that $\phi_N(x_0)\ge 0$ and $\psi_N(x_0)\ge 0$. 
From \eqref{eq:phiNext} we get first
\begin{equation}\label{eq:boundbyc1}
\phi(x) \le c(x,y_0), \quad \psi(y) \le c(x_0,y) 
\end{equation}
and then 
\begin{equation}\label{eq:boundbyc2}
\phi(x) \ge \inf_{y} \{c(x,y) - c(x_0,y)\}\,,  \quad \psi(x) \ge  \inf_{x} c(x,y) - c(x,y_0) \,.  
\end{equation}
In particular we have $\|\phi\|_\infty \le \|c\|_\infty$ and $\|\psi\|_\infty \le \|c\|_\infty$.  Since $P_N$ and $Q_N$ converge weakly to $P$ and $Q$ by choosing $N$ large enough we have then form \eqref{eq:tt23}  
\begin{equation}
\begin{aligned}
\mathfrak{D}^c_\varepsilon(P\|Q)  &\le \mathbb{E}_{P}[\widetilde{\phi}] - \varepsilon\ln \mathbb{E}_{Q}[e^{-\widetilde{\psi}/\varepsilon}] + 2 \alpha \\
&\le \sup_{(\phi,\psi)\in \Phi_c}\left\{ \mathbb{E}_{P}[\widetilde{\phi}] - \varepsilon\ln \mathbb{E}_{Q}[e^{-\widetilde{\psi}/\varepsilon}]\right\}  + 2 \alpha 
\end{aligned}
\end{equation} 
Since $\alpha$ is arbitrary this concludes the proof of the theorem.
\end{proof}

Finally, we apply an additional approximation argument to handle general lower semi-continuous cost functions.

\begin{theorem}\label{thm:duality.general}
Assume that $X$ and $Y$ are Polish spaces and that $c$ is bounded below and lower semicontinuous and that $c \le a \oplus b$ with $a \in L^1(P)$ and $b \in L^1(Q)$.  Then duality holds 
\begin{equation}
\mathfrak{D}^c_\varepsilon(P\|Q) = \sup_{(\phi,\psi)\in \Phi_c} \left\{ \mathbb{E}_P[\phi] - \varepsilon\ln \mathbb{E}_Q[e^{-\psi/\varepsilon}] \right\}
\end{equation}
\end{theorem}

\begin{proof} The condition 
$c \le a \oplus b$ with $a \in L^1(P)$ implies that $\mathfrak{D}^c_\varepsilon(P\|Q)$ is finite since 
$$ 
\mathfrak{D}^c_\varepsilon(P\|Q) \le T_c(P,Q)
\le \inf_{\pi \in \Pi(P,Q)} \int c d\pi \le \mathbb{E}_P[a] + \mathbb{E}_Q[b] <\infty .
$$

We can write the lower semicontinuous $c$ as an increasing limit of nonnegative continuous function $c_n$ and by replacing $c_n$ by $\max\{c_n,n\}$ we can assume that the functions $c_n$ are bounded and continuous.

We already know that
\begin{equation}
\mathfrak{D}^{c}_\varepsilon(P\|Q) \ge \sup_{(\phi,\psi)\in \Phi_{c}} \left\{ \mathbb{E}_P[\phi] - \varepsilon\ln \mathbb{E}_Q[e^{-\psi/\varepsilon}] \right\} .
\end{equation}

By theorem \ref{thm:duality.bounded}
and because $c_n$ is increasing to $c$ we have 
\begin{equation}
\begin{aligned}
\mathfrak{D}^{c_n}_\varepsilon(P\|Q)&=
\sup_{(\phi,\psi)\in \Phi_{c_n}} \left\{ \mathbb{E}_P[\phi] - \varepsilon\ln \mathbb{E}_Q[e^{-\psi/\varepsilon}] \right\}
\\
&\le \sup_{(\phi,\psi)\in \Phi_{c}} \left\{ \mathbb{E}_P[\phi] - \varepsilon\ln \mathbb{E}_Q[e^{-\psi/\varepsilon}] \right\},
\end{aligned}
\end{equation}
and therefore it is enough to show that 
\begin{equation}
\lim_{n \to \infty}\mathfrak{D}^{c_n}_\varepsilon(P\|Q)\ge \mathfrak{D}^{c}_\varepsilon(P\|Q) \,.
\end{equation}
Since $c_n$ is increasing then the sequence $\mathfrak{D}^{c_n}_\varepsilon(P\|Q)$ is increasing in $n$ and so the limit exists.  We have, using the optimal measure $R_n$ and an optimal transport plan $\pi_n \in \Pi(P,R_n)$  
\begin{equation}
\mathfrak{D}^{c_n}_\varepsilon(P\|Q) = \int c_n d\pi_n + \varepsilon \mathfrak{D}_{\textrm{KL}}(R_n \|Q)\,.
\end{equation}
We have 
\begin{equation}
\varepsilon \mathfrak{D}_{\textrm{KL}}(R_n \|Q) \le \mathfrak{D}^{c_n}_\varepsilon(P\|Q)  \le \mathfrak{D}^{c}_\varepsilon(P\|Q) \,.
\end{equation}
and therefore the sequence $\{R_n\}$ is tight as well.  This implies that the sequence of optimal transport plans $\{\pi_n\}$ is tight and therefore 
there exist accumulation points $\pi^*,R^*$ with $\pi^* \in \Pi(P,R^*)$.  Since $c_n$
is increasing we have for $n\ge m  $
\begin{equation}
\begin{aligned}
\mathfrak{D}^{c_n}_\varepsilon(P\|Q) \ge \int c_m d\pi_n + \varepsilon \mathfrak{D}_{\textrm{KL}}(R_n \|Q)
\end{aligned}
\end{equation}
and thus for any $m$
\begin{equation}
\begin{aligned}
\lim_{n \to \infty} \mathfrak{D}^{c_n}_\varepsilon(P\|Q) &\ge \limsup_{n \to \infty} \int c_m d\pi_n + \varepsilon \mathfrak{D}_{\textrm{KL}}(R_n \|Q) \\
&\ge \int c_m d\pi^* + \varepsilon \mathfrak{D}_{\textrm{KL}}(R^* \|Q)
\end{aligned}
\end{equation}
By monotone convergence $\lim_{m\to \infty}\int c_m d\pi^* = \int c d\pi^*$ and thus 
\begin{equation}
\lim_{n \to \infty} \mathfrak{D}^{c_n}_\varepsilon(P\|Q) \ge \int c d\pi^* +  \varepsilon \mathfrak{D}_{\textrm{KL}}(R^* \|Q) \ge \mathfrak{D}^{\varepsilon}_c(P\|Q)\,.
\end{equation}
This concludes the proof.

\end{proof}

Finally we prove the existence of optimizers in the general case.  Since $c$ is only lower semi-continuous we will now take $\phi \in L^1(P)$ and $\psi \in L^1(Q)$ instead of bounded and continuous. 
We will use a weak convergence argument in $L^1$, due to  \cite{rachev1998mass} (see also  \cite{Villani2003topics}) and which me modify accordingly.

\begin{theorem}\label{thm:maximizer.generalcase}
Assume that $X$ and $Y$ are Polish spaces and that $c$ is bounded below and lower semicontinuous and that $c \le a \oplus b$ with $a \in L^1(P)$ and $b \in L^1(Q)$.  Then the dual problem 
\begin{equation}
\mathfrak{D}^c_\varepsilon(P\|Q) = \sup_{ \phi\in L^1(P),  \psi\in L^1(Q) \atop \phi \oplus \psi \le c} \left\{ \mathbb{E}_P[\phi] - \varepsilon\ln \mathbb{E}_Q[e^{-\psi/\varepsilon}] \right\}
\end{equation}
possesses a maximizer $(\phi^*, \psi^*)$.
\end{theorem}

\begin{proof} Our assumption on $c$ insures that $\mathfrak{D}^c_\varepsilon(P\|Q)$ is finite and let us consider a optimizing sequence $(\phi_n,\psi_n)$ such that 
\begin{equation}
    \lim_{n \to \infty} \left(\mathbb{E}_P[\phi_n] -\varepsilon \log \mathbb{E}_Q[e^{-\varepsilon \psi_n/\varepsilon}]\right)= \mathfrak{D}^c_\varepsilon(P \|Q)
\end{equation}

As usual we can assume that $\phi_n=\psi_n^c$ and therefore we get the upper bounds
\begin{equation}
\begin{aligned}
\phi_n(x) &\le c(x,y_0)- \psi_n(y_0) \le a(x) +b(y_0)-\psi_n(y_0) \equiv\tilde{a}(x) \in L^1(P) \\
\psi_n(x) &\le c(x_0,y)- \phi_n(x_0) \le b(Y) +a(x_0)-\phi_n(x_0) \equiv\tilde{b}(y) \in L^1(Q) \\
\end{aligned}
\end{equation}
and so the sequence $\phi_n$ and $\psi_n$ are bounded above by $L^1$ function.  

Since $(\phi_n,\psi_n)$ is an optimizing sequence and using Jensen inequality \begin{equation}
\mathfrak{D}^c_\varepsilon(P\|Q) - \alpha< \mathbb{E}_P[\phi_n] - \varepsilon\ln \mathbb{E}_Q[e^{-\psi_n/\varepsilon}] \le \mathbb{E}_P[\phi_n] + E[\psi_n] \le T_c(P,Q) < \infty
\end{equation}
Together with the bound on $\phi_n$ and $\psi_n$ this implies that $\phi_n \in L^1(P)$ and $\psi_n \in L^1(Q)$.  Moreover since we can always replace $(\phi_n,\psi_n)$ by $(\phi_n + a_n, \psi_n-a_n)$ we can assume that the sequence $\{\phi_n\}$ and $\{\psi_n\}$  are bounded in $L^1(P)$ and $L^1(Q)$ respectively.

We show next that $\phi_n$ and $\psi_n$ have suitable limits. To do this let us set 
\begin{equation}
f_n = \tilde{a} - \phi_n \ge 0 \quad g_n = \tilde{b} -\psi_n \ge 0
\end{equation}
and $f_n \in L^1(P)$ and $g_n \in L^1(Q)$.  We truncate $f_n$ and $g_n$ by setting 
\begin{equation}
f_n^{(l)}= \max \{ f_n, l\} \quad  g_n^{(l)}= \max \{ g_n, l\} 
\end{equation}
Note that 
\begin{equation}
\begin{aligned}
0\le f_n^{(l)} \le l & \quad 
\text {and} \quad  f_n^{(l)} \le f_n^{(l+1)} \le \cdots \\
0\le g_n^{(l)} \le l & \quad 
\text {and} \quad g_n^{(l)} \le g_n^{(l+1)} \le \cdots 
\end{aligned}
\end{equation}
If we now set 
$\phi_n^{(l)} = \tilde{a} - f_{n}^{(l)}$ and  
$\psi_n^{(l)} = \tilde{b} - f_{n}^{(l)}$ we have 
\begin{equation}
\begin{aligned}
\tilde{a} &\ge \cdots \ge \phi_n^{(l)}\ge \phi_n^{(l+1)} \ge \cdots \ge \phi_n
\\
\tilde{b} &\ge \cdots \ge  \psi_n^{(l)}\ge \psi_n^{(l+1)} \ge \cdots \ge \psi_n
\end{aligned}
\end{equation}
and in particular
\begin{equation}\label{eq:mon}
\mathbb{E}_P[\phi_n^{(l)}] - \varepsilon \ln \mathbb{E}_Q[e^{-\psi_n^{(l)}/\varepsilon}]
\ge \mathbb{E}_P[\phi_n] - \varepsilon \ln \mathbb{E}_Q[e^{-\psi_n/\varepsilon}]
\end{equation}

Moreover since 
\begin{equation}
\phi_n^{(l)}= \tilde{a} + O(l)
\end{equation}
the sequence $\{\phi_n^{(l)}\}_{n \ge 1}$ is uniformly integrable in $L^1(P)$, hence weakly compact.  Up to extraction of a subsequence we can thus assume that $\phi_n^{(l)}$ converges to weakly in $L^1$ to some $\phi^{(l)}\in L^1(P)$.  By using a diagonal sequence extraction we can then assume that $\phi_n^{(l)}$ converges to $\phi^{(l)}$ for every $l$. Since weak convergence preserve ordering  we have $\tilde{a} \ge \phi^{(l)}\ge \phi^{(l+1)} \ge \cdots$.   
The same results holds of course for $\psi_n^{(l)}$.  

Recall the weak lower semicontinuity of convex functions  \cite{brezis1999analyse} (that is if $u_n$ converges weakly in $L^1$ to $u$ and $F$ is convex then $\int F(u) dQ \le \liminf_{n} \int F(u_n) dQ$). 
From this and from \eqref{eq:mon} we obtain
\begin{equation}
\begin{aligned}
\mathbb{E}_P[\phi^{(l)}] - \varepsilon \ln \mathbb{E}_Q[e^{-\psi^{(l)}/\varepsilon}] 
& \ge  \limsup_{n}\left( \mathbb{E}_P[\phi_{n}^{(l)}] - \varepsilon \ln \mathbb{E}_Q[e^{-\psi_{n_k}^{(l)}/\varepsilon}] \right) \\
& \ge  \limsup_{n}\left( \mathbb{E}_P[\phi_{n}] - \varepsilon \ln \mathbb{E}_Q[e^{-\psi_{n}/\varepsilon}] \right) \\
& = \mathfrak{D}^c_\varepsilon(P\|Q).
\end{aligned}
\end{equation}

Finally we note that the sequences $\phi^{(l)}$ and $\psi^{(l)}$ are bounded in $L^1$, are decreasing in $l$, and are bounded above by integrable functions. Therefore we can apply the monotone convergence 
and deduce the existence of  $L^1$ limits $\phi$ and $\psi$, defined almost everywhere 
\begin{equation}
\phi=\lim_{l\to \infty} \phi^{(l)}, \quad \psi=\lim_{l \to \infty} \psi^{(l)}
\end{equation}
By the monotone convergence theorem we obtain that 
\begin{equation}
\mathbb{E}_P[\phi] -\varepsilon \ln \mathbb{E}_Q[e^{-\psi/\varepsilon}] = \lim_{l \to \infty} \mathbb{E}_P[\phi^{(l)}] - \varepsilon \ln \mathbb{E}_Q[e^{-\psi^{(l)}/\varepsilon}] \ge 
\mathfrak{D}^c_\varepsilon(P \|Q)
\end{equation}
\end{proof}

\section{Proof of DPI and additivity in Section \ref{sec:DPI} }
\label{sec:proofs:DPI}

We first give the proof of the invariance of the proximal $p$-Wasserstein divergence under isomorphism.

\noindent
{\it Proof of Theorem \ref{thm:isomorphism.invariance}}  
We use the dual variational representation.  Since $c(Tx,Ty)=c(x,y)$, $\phi \oplus \psi \le c$ implies that $\phi \circ T \oplus \psi \circ T \le c$. Therefore 
\begin{equation}
\begin{aligned}
 \mathfrak{D}^{p}_{\textrm{KL},\varepsilon}( T_\# P\| T_\# Q) &= \sup_{\phi \oplus \psi \le c} \mathbb{E}_P[\phi \circ T] - \varepsilon \ln \mathbb{E}_Q[e^{-\psi\circ T/\varepsilon}] \\
 &\le \sup_{\phi'  \oplus \psi' \le c} \mathbb{E}_P[\phi' ] - \varepsilon \ln \mathbb{E}_Q[e^{-\psi'/\varepsilon}]\\
 & = \mathfrak{D}^{p}_{\textrm{KL},\varepsilon}(  P\| Q).
\end{aligned}
\end{equation}
We have inequality since in the first supremum we only consider $\phi'$ and $\psi'$ which have the form $\phi'=\phi \circ T$ and $\psi'=\psi\circ T$. Since $T$ is 
assumed to be invertible however, we obtain equality.
\qed

The proof of the data processing inequality in Theorem \ref{thm:dpi} follows by a similar argument and is therefore omitted. Finally we prove the additivity  
property of the proximal OT divergence. 
\medskip

\noindent{\it Proof of Theorem \ref{thm:additivity}}
Since $c(x,y)=\|x-y\|_p^p=\|x_1-y_1\|_p^p + \|x_2-y_2\|_p^p=c_1(x_1,y_1)+c_2(x_2,y_2)$ let us use the primal representation for optimal cost and for the proximal OT-divergence and restrict both infima to product measures and product coupling.  By using the additivity of relative entropy we obtain then the lower bound
\begin{equation}
\begin{aligned}
& \mathfrak{D}^{p}_{\textrm{KL},\varepsilon}( P_1 \times P_2  \| Q_1\times Q_2) \\ &\le \inf_{R=R_1\times R_2} \left\{W_p^p( P_1 \times P_2, R_1 \times R_2) + \varepsilon \mathfrak{D}(R_1 \times R_2\| Q_1\times Q_2) \right\}\\
&\le \inf_{ R_1, R_2}  \inf_{\pi_1 \in \Pi(P_1,R_1) \atop
\pi_2 \in \Pi(P_2,R_2) } 
\bigg{\{}\int_{X_1} \|x_1-y_1\|_p^p  d\pi_1 + \int_{X_2} \|x_2-y_2\|_p^p  d\pi_2  \\
& \hspace{5cm}  + \varepsilon \mathfrak{D}(R_1 \|Q_1) + \varepsilon \mathfrak{D}( R_2\|  Q_2) \bigg{\}} \\
&=   \mathfrak{D}^{p}_{\textrm{KL},\varepsilon}( P_1  \| Q_1 ) +   \mathfrak{D}^{p}_{\textrm{KL},\varepsilon}( P_2  \| Q_2 )
\end{aligned}
\end{equation}
On the other hand using the dual variational representation and considering  only  $\phi$ and $\psi$  of the form $\phi(x)=\phi_1(x_1)+ \phi_2(x_2)$ and 
$\psi(y)=\psi_1(y_1)+ \psi_2(x_2)$  with  $\phi_1(x_1)+\psi_1(x_1) \le \|x_1-y_1\|_p^p$ and $\phi_2(x_2)+\psi_2(x_2) \le \|x_2-y_2\|_p^p$. We have then the upper bound
\begin{equation}
\begin{aligned}
& \mathfrak{D}^{p}_{\textrm{KL},\varepsilon}( P_1 \times P_2  \| Q_1\times Q_2) \\ &\ge \sup_{\phi_1 \oplus \psi_1 \le c_1 \atop \phi_2 \oplus \psi_2 \le c_2  } \left\{ \mathbb{E}_{P_1 \times P_2}[\phi_1 + \phi_2]  - \varepsilon
\log \mathbb{E}_{Q_1 \times Q_2} [ e^{-(\psi_1+\psi_2)/\varepsilon}]\right\} 
\\
&= \sup_{\phi_1 \oplus \psi_1 \le c_1 } \left\{ \mathbb{E}_{P_1}[\phi_1]   - \varepsilon
\log \mathbb{E}_{Q_1}  [ e^{-(\psi_1)\varepsilon}] \right\} \\
& \hspace{1cm} + \sup_{ \phi_2 \oplus \psi_2 \le c_2  } \left\{ \mathbb{E}_{P_2}[ \phi_2]  - \varepsilon
\log \mathbb{E}_{Q_2} [ e^{-\psi_2/\varepsilon}] \right\}
\\
&=   \mathfrak{D}^{p}_{\textrm{KL},\varepsilon}( P_1  \| Q_1 ) +   \mathfrak{D}^{p}_{\textrm{KL},\varepsilon}( P_2  \| Q_2 )
\end{aligned}
\end{equation}
This concludes the proof. \qed

\section{Proof of Theorem \ref{thm:first variation}}\label{sec:proofs:firstvariation}
We use the Jordan decomposition $\rho=\rho_+-\rho_-$ where $\rho_{\pm}\in \mathcal{P}(X)$ are mutually singular so there exists two disjoint sets $X_\pm$ such that $\rho_\pm(A)=\rho_\pm(A \cap X_\pm)$
for all measurable sets $A$. The measure $P+ \alpha(\rho_+-\rho_-)$ has total mass $1$ but to be a probability measure we need that
$\alpha \rho_-(A) \le (P+\alpha \rho_+)(A)$ holds for all $A$. This implies that $\rho_-$ is absolutely continuous with respect to $P$.  Indeed if $P(A)=0$ then 
\begin{equation}\label{eq:perturb:abs_contin}
\alpha \rho_-(A)=\alpha \rho_-(A \cap X_-) \le P(A\cap X_-) + \alpha \rho_+(A \cap X_-) \le P(A) =0.   
\end{equation}

By Theorem \ref{thm:divergence} the map 
$F(\alpha)= D_{c,\textrm{KL}}^\varepsilon(P + \alpha \rho\|Q)$ is lower semicontinuous and convex. By assumption it is also finite (and thus also continuous) on the interval $[0,\alpha_0)$  As a 
finite convex function on an interval the map $F(\alpha)$ has left derivative $F'_-(\alpha)$ and right derivative $F'_+(\alpha)$ at every point $\alpha \in (0,\alpha_0)$ and the right derivative $F_+'(0)$ exists. 
Moreover the derivative $F'(\alpha)$ exists  and $\lim_{\alpha \searrow 0} F_-'(\alpha)=F'_+(0)$.

For $0< \alpha < \alpha_0$ let us denote $(\phi^*_\alpha, \psi^*_{\alpha})$ the optimizers for 
the dual representation of $D_{c,\textrm{KL}}^\varepsilon(P + \alpha \rho\|Q)$  which are unique up to constant in $\textrm{supp}{P + \rho_+}$ and $\textrm{supp}(Q)$ respectively.  For $h$ sufficiently small we have 
\begin{equation}
\begin{aligned}
F(\alpha \pm h)&= \sup_{(\phi, \psi) \in \Phi_c} \left\{ \mathbb{E}_{P+(\alpha\pm h)\rho}[\phi] - \varepsilon \log \mathbb{E}_{Q}[e^{-\psi/\varepsilon}] \right\} \\
& \ge \mathbb{E}_{P+(\alpha\pm h)\rho}[\phi^*_\alpha] -  \varepsilon \log \mathbb{E}_{Q}[e^{-\psi^*_\alpha/\varepsilon}]
\\
&= F(\alpha)  \pm  h \int \phi^*_\alpha d \rho 
\end{aligned}
\end{equation}
from which we conclude that 
\begin{equation}\label{eq:falphapm}
F'_{-}(\alpha) \le \int \widehat{\phi}^*_\alpha d \rho  \le F'_{+}(\alpha)\,.
\end{equation}
For $\alpha=0$ a similar computation using the optimizer $\widehat{\phi}^*$ defined in \eqref{eq:phistarextension} we find that 
\begin{equation}\label{eq:fprimepluslower}
F'_{+}(0) \ge \int  \widehat{\phi}^* d \rho \,.
\end{equation}
Let us now pick a sequence $\alpha_n \searrow  0$ such that $F$ is differentiable at all $\alpha_n$. By \eqref{eq:falphapm} we have 
\begin{equation}
F'_+(0)=\lim_{n\to \infty} F'(\alpha_n) = \lim_{n \to \infty}\int \widehat{\phi}^*_{\alpha_n} d \rho
\end{equation}

Since $c$ is bounded, arguing as in Theorem \ref{thm:duality.bounded} (see equations \eqref{eq:boundbyc1} and \eqref{eq:boundbyc2}) 
the sequences $(\widehat{\phi}^*_{\alpha_n}$ and $\psi^*_{\alpha_n})$
are  equibounded (by $\|c\|_\infty$).   

Since $c$ is uniformly continuous it has a modulus of continuity, i.e. $|c(x,y)=c(x',y')|\le \omega(\mathfrak{D}(x,x')+ d(y,y'))$
for some function $\omega: \mathbb{R}^+\to \mathbb{R^+}$ which satisfies $\omega(0)=0$ and is continous at $0$. Since we can assume that $\widehat{\phi}^*_{\alpha_n}$ and $\psi^*_{\alpha_n}$ are $c$-transform of each other, they inherit the modulus of continuity from $c$. By an Arzela-Ascoli type argument this implies that the sequences $\widehat{\phi}^*_{\alpha_n}$ and 
$\psi^*_{\alpha_n}$ converge, uniformly on compact sets, to some functions $\widetilde{\phi}$ and $\widetilde{\psi}$.    
We have then, by lower semicontinuity,
\begin{equation}
\begin{aligned}
 \mathfrak{D}^{c}_{\textrm{KL},\varepsilon}(P \|Q) &= \lim_{n \to \infty}
 \mathfrak{D}^{c}_{\textrm{KL},\varepsilon}(P + \alpha_{n} \rho\|Q) \\
& =\lim_{n \to \infty} 
\left\{ \mathbb{E}_{P + \alpha_n\rho}[\widehat{\phi}^*_{\alpha_n}] - \varepsilon \log \mathbb{E}_Q[e^{-\psi^*_{\alpha_n}/\varepsilon}] \right\} \\
& =  \mathbb{E}_P[\widetilde{\phi}] - \varepsilon \log \mathbb{E}_Q[e^{-\widetilde{\psi}/\varepsilon}] \\
& \le  \mathfrak{D}^{c}_{\textrm{KL},\varepsilon}(P \|Q)
\end{aligned}
\end{equation}
and thus $(\widetilde{\phi}, \widetilde{\psi})$ must be an optimizer.  By definition  $\widetilde{\phi}(x)=\widehat{\phi}^*(x) $ for $x \in \textrm{supp}(P)$ and $\widetilde{\phi}(x)\le \widehat{\phi}^*(x)$ for all $x\in X$. 
Therefore, using that $\rho_-$ is absolutely continuous with respect to $P$ 
\begin{equation}
F'_+(0)= \int \widetilde{\phi}(x) d\rho = \int \widetilde{\phi}(x) d\rho_+ - \int \widetilde{\phi}(x) d\rho_- \le \int  \widehat{\phi}^* d\rho 
\end{equation}
Combine with \eqref{eq:fprimepluslower} this shows that $F'_+(0)=\int  \widehat{\phi}^* d\rho$.

\section{Example: Proximal OT divergences for Gaussians}
\label{sec:Gaussians}

In this section, we derive an explicit expression for the proximal 2-Wasserstein divergence between two Gaussian distributions. We begin with the one-dimensional case and then extend the analysis to the multivariate setting with commuting covariance matrices.

\begin{theorem}\label{thm:1dgaussian}
    Let $P = \mathcal{N}(m_1,\sigma_1^2)$ and  $Q = \mathcal{N}(m_2,\sigma_2^2)$. 
    Then the OT-divergence made by 2-Wasserstein and the KL-divergence is given by choosing the intermediary distribution to be the Gaussian distribution 
    $R = \mathcal{N}(m_R,\sigma_R^2)$ with
\begin{eqnarray}
    m_R &=& \frac{m_1 + \frac{\varepsilon}{2 \sigma_2^2} m_2}{1 +\frac{\varepsilon}{2 \sigma_2^2} } \label{eq:meanR} \\
    \sigma_R &=& \frac{\sigma_1 + \sqrt{ \sigma_1^2 + 2 \varepsilon(1 + \frac{\varepsilon}{2 \sigma_2^2}) }}{2(1 +\frac{\varepsilon}{2 \sigma_2^2})}\label{eq:sigmaR}
\end{eqnarray}    
and the proximal OT divergence takes the form 
\begin{equation}\label{eq:d2r}
\mathfrak{D}^2_{\mathrm{KL},\varepsilon}(P\|Q) = \frac{\frac{\epsilon}{2 \sigma_2^2}}{1+\frac{\epsilon}{2 \sigma_2^2}}(m_1-m_2)^2 + \sigma_1 (\sigma_1-\sigma_R) + \varepsilon \ln\frac{\sigma_{2}}{\sigma_R}
\end{equation}
\end{theorem}

\noindent We illustrate Theorem \ref{thm:1dgaussian}
in Figure \eqref{fig:1dgaussian} by displaying the intermediate measure $R=R_\varepsilon$ for several values of $\varepsilon$.

\begin{figure}[h] 
\centering 
\includegraphics[width=0.6\textwidth]{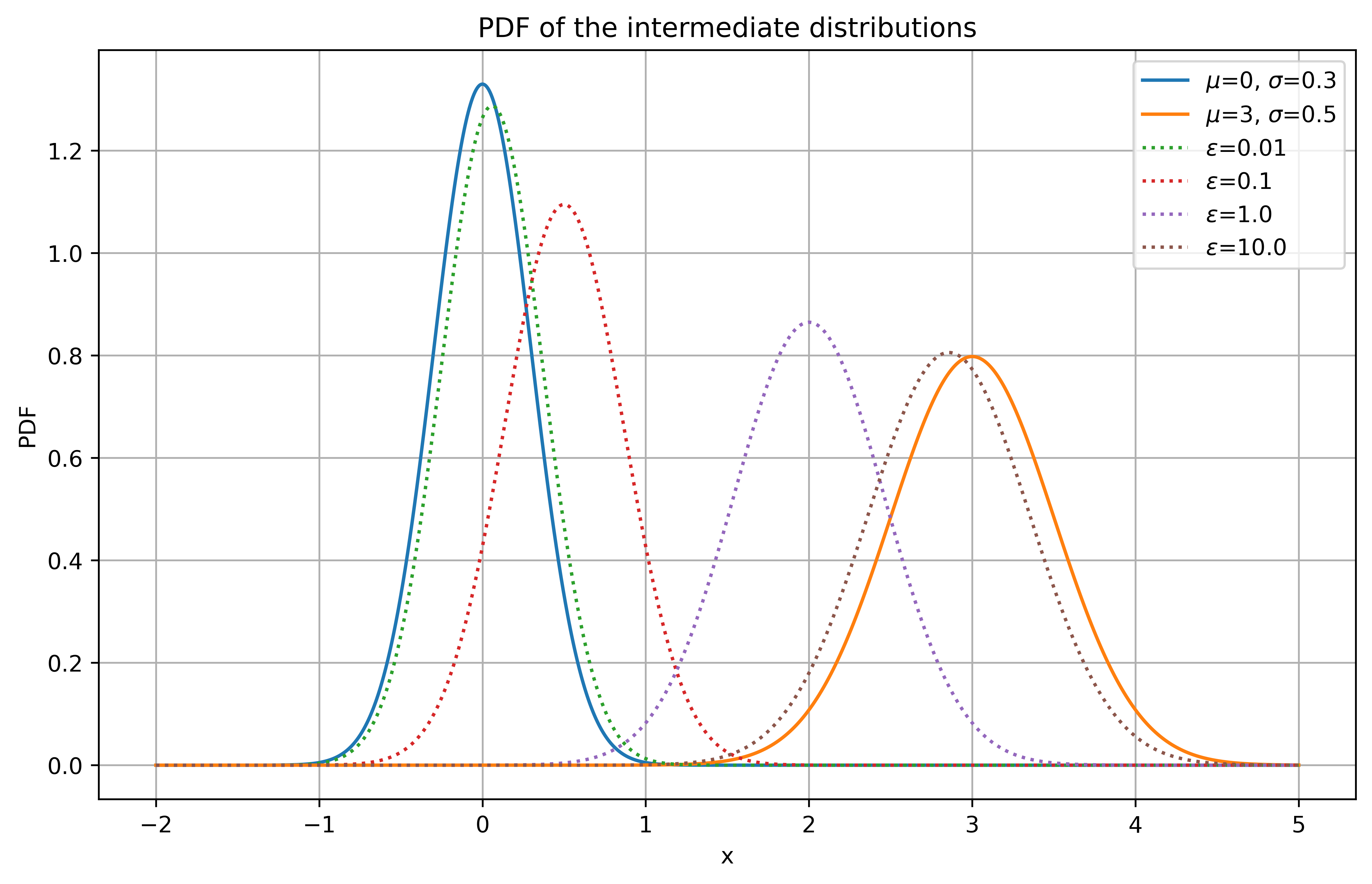} 
\caption{Intermediate measure for several values of $\varepsilon$ for 1-d Gaussian.} 
    \label{fig:1dgaussian} 
\end{figure}

\begin{proof}
The idea is that 
the intermediate measure $R$ should be Gaussian
itself.  If this is the case, using the explicit formulas for the 2-Wasserstein distance and the KL-divergence between Gaussian distributions we find 
\begin{equation}
\begin{aligned}
& W_2^2(P,R) + \varepsilon \mathfrak{D}_{\mathrm{KL}}(R\|Q) \\
&= (m_1 - m_R)^2 + (\sigma_1 - \sigma_R)^2 + 
\varepsilon \left(
\ln\frac{\sigma_2}{\sigma_R} + \frac{1}{2}\left( \frac{\sigma_R^2 +(m_R - m_2)^2}{\sigma_2^2} - 1
\right) \right)\\
\end{aligned}
\end{equation}
Optimizing over $m_R$ and $\sigma_R$ along with straightforward computations, yields the expressions in \eqref{eq:meanR}--\eqref{eq:d2r}. To establish that the infimum is attained by an intermediate Gaussian distribution, we rely on the dual variational representation \eqref{eq:OTdivergence:dual}To get a lower bound on the OT divergence we take the infimum over gaussian  $R=\mathcal{N}(m_R,\sigma_R)$.  We denote by $\tilde{P} = \mathcal{N}(0,\sigma_1)$
the shifted version of $P$
and similarly  for $\tilde{Q}$ and $\tilde{R}$.

We have 
\begin{equation}
\begin{aligned}
& W_2^2(P,R) + \varepsilon \mathfrak{D}_{\mathrm{KL}}R\|Q) \\
&= (m_1 - m_R)^2 + (\sigma_1 - \sigma_R)^2 + 
\varepsilon \left(
\ln\frac{\sigma_2}{\sigma_R} + \frac{1}{2}\left( \frac{\sigma_R^2 +(m_R - m_2)^2}{\sigma_2^2} - 1
\right) \right)\\
& = (m_1 - m_R)^2 + \varepsilon \frac{(m_R - m_2)^2}{\sigma_2^2}
+ W_2^2(\tilde{P},\tilde{R}) +\varepsilon \mathfrak{}\tilde{R}\|\tilde{Q})
\end{aligned}
\end{equation}
as a sum of two non-negative terms and so we can optimize over the mean $m_R$ and the variance $\sigma_R$ separately. Optimizing over $m_R$ gives \eqref{eq:meanR}
and 
\[
(m_1 - m_R)^2 + \varepsilon \frac{(m_R - m_2)^2}{\sigma_2^2}= \frac{\frac{\epsilon}{2 \sigma_2^2}}{1+\frac{\epsilon}{2 \sigma_2^2}}(m_1-m_2)^2
\]

Minimzing over $\sigma_R$ gives the
stationarity condition
\begin{equation}\label{eq:stationarity1}
2 (\sigma_R - \sigma_1) - \varepsilon \frac{1}{\sigma_R} + \varepsilon \frac{\sigma_R}{\sigma_2^2} =0 
\end{equation}
This gives a quadratic equation, and taking the positive root we find 
\eqref{eq:sigmaR}.
Note also that the stationarity condition can be rewritten as 
\begin{equation}\label{eq:stationarity2}
\frac{\varepsilon}{2}\left( \frac{\sigma_R^2}{\sigma_2^2} -1\right) = \sigma_1 \sigma_R - \sigma_R^2
\end{equation}
and so we find that for the optimal $\sigma_R$ we have 
\begin{equation}\label{eq:optimalD}
W_2^2(\tilde{P},\tilde{R})+ \varepsilon \mathfrak{} \tilde{R}\|\tilde{Q})
=  \sigma_1 (\sigma_1 - \sigma_R) + \epsilon \ln \frac{\sigma_2}{\sigma_R} \,.
\end{equation}

To find an upper bound on the OT divergence we use the dual representation. Since we are dealing with Gaussian  it is natural to ake $\phi(x)$ and $\psi(y)$ to be quadratic polynomial and since we can assume that the mean is $0$ we can simply take quadratic. In view of Theorem \ref{thm:duality.Polish} 
we choose to write $\psi$ as
\begin{equation}
\psi(y) =  \varepsilon\left(\frac{y^2}{2 \sigma_R^2} -\frac{y^2}{2 \sigma_2^2}
\right)\end{equation}
Next, we show that if $\sigma_R$ satisfies the stationary condition \eqref{eq:stationarity1} 
then 
\[
E_{\tilde{P}}[\psi^c(x)] - \varepsilon \ln \mathbb{E}_{\tilde{Q}}[ e^{-\psi(y)/\varepsilon}]= 
W_2^2(\tilde{P},\tilde{R}) + \varepsilon \mathfrak{}\tilde{R}\|\tilde{Q})
\]
holds which completes the proof.  First note that 
\[
E_{\tilde{Q}}[e^{-\psi/\varepsilon}]
= \frac{1}{\sqrt{2 \pi \sigma_2^2}} \int e^{-y^2/2 \sigma_{R}^2} dy = \varepsilon\ln\frac{\sigma_R}{\sigma_2}
\]
so that $-\varepsilon \log \mathbb{E}_{\tilde{Q}}[e^{-\psi/\varepsilon}] = \varepsilon \log \frac{\sigma_2}{\sigma_R}$ which is the second term in \eqref{eq:optimalD}. 
Computing the $c$ transform of $\psi(y)=a y^2$ we find 
\begin{equation}
\psi^c(x)= \inf_{y}\{ (x-y)^2 - a y^2\} = x^2 - \frac{1}{1-a} x^2 
\end{equation}
for our choice of $a$, if we use the stationarity condition in the form of \eqref{eq:stationarity2}  
\[
a = \varepsilon\left(\frac{1}{2 \sigma_R^2} -\frac{1}{2 \sigma_2^2}
\right) = \frac{1}{\sigma_R^2} \frac{\epsilon}{2}\left( 1 - \frac{\sigma_R^2}{\sigma_2^2} \right)=\frac{\sigma_R-\sigma_1}{\sigma_R}
\]
and thus $1-a = \frac{\sigma_1}{\sigma_R}$. Using this we obtain
\[
E_{\tilde{P}}[\psi^c(x)]= \sigma_1^2 - \frac{1}{1-a} \sigma_1^2
=\sigma_1^2 - \sigma_R \sigma_1
\]
which is the first term in  \eqref{eq:optimalD}. 
This concludes the proof.

\end{proof}

For multivariate Gaussians with commuting covariance matrices, the sequence of calculations from the one-dimensional case extends directly. The general case of non-commuting covariances is left for future work.

\begin{theorem}\label{thm:commuting.gaussian}
    Let $P = \mathcal{N}(m_1,\Sigma_1)$ and $Q = \mathcal{N}(m_2,\Sigma_2)$, where $\Sigma_1$ and $\Sigma_2$  commute. Then the proximal OT divergence $\mathfrak{D}^2_{\mathrm{KL},\varepsilon}$
    is given by choosing the intermediary distribution to be the Gaussian
    $R = \mathcal{N}(m_R,\Sigma_R)$ with 
    \begin{equation}
    \begin{aligned}
    m_R &= \left( I + \frac{ \varepsilon}{2}\Sigma_2^{-1}\right)^{-1}\left( m_1+ \frac{\varepsilon}{2}\Sigma_2^{-1}\mu_2 + \right) \\
    \Sigma_R^{1/2}&= \left( I + \frac{ \varepsilon}{2}\Sigma_2^{-1}\right)^{-1} \frac{1}{2}
    \left( \Sigma_1^{1/2} + 
    \left( \Sigma_1 + 2\varepsilon \left( I + \frac{ \varepsilon}{2}\Sigma_2^{-1}\right)^{-1}\right)^{1/2}
    \right) 
 .  \end{aligned}
    \end{equation}
\end{theorem}

\begin{proof} Since $\Sigma_1$ and $\Sigma_2$ commute, we find an orthogonal transformation which diagonalize $\Sigma_1$ and $\Sigma_2$ simultaneously.  there exists an orthogonal transformation that simultaneously diagonalizes both matrices. By Theorem \ref{thm:isomorphism.invariance}, we may therefore assume without loss of generality that both $\Sigma_1$ and $\Sigma_2$ are diagonal, implying that $P$ and $Q$ are product measures. The result then follows from the additivity property (Theorem \ref{thm:additivity}) combined with the one-dimensional case established in Theorem \ref{thm:1dgaussian}.
\end{proof}

\begin{figure}[h] 
\centering 
\includegraphics[width=0.6\textwidth]{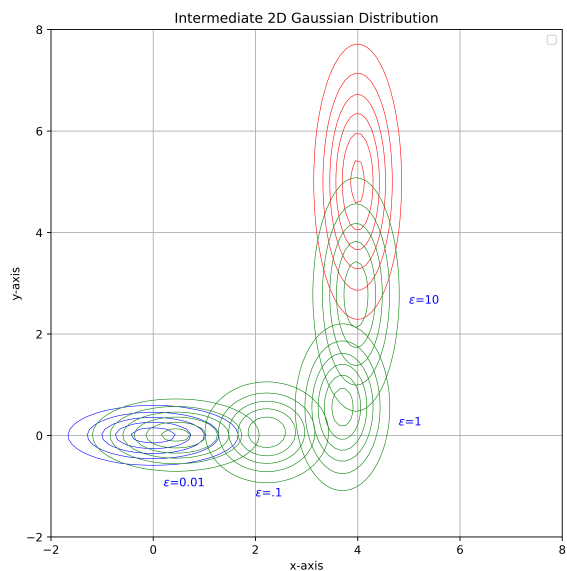} 
\caption{Intermediate measures for 2D Gaussian for several values of $\varepsilon$.} 
    \label{fig:2dgaussian} 
\end{figure}

\end{document}